\documentclass[reqno]{amsart}
\usepackage{amssymb,eucal,latexsym, mathrsfs,xy,enumerate,graphicx,url}

\xyoption{all}

{\end{enumerate}}
{\end{enumerate}}
\newenvironment{enumerater}{\begin{enumerate}[\upshape (1)]}%
{\end{enumerate}}

\usepackage[usenames]{color}

\newcommand{\pup}[1]{\textup{(}{#1}\textup{)}}
\newcommand{\fring}{\textit{f}-ring}

\newcommand{\lgrp}{$\ell$-group}
\newcommand{\lhom}{$\ell$-ho\-mo\-mor\-phism}
\newcommand{\lrep}{$\ell$-rep\-re\-sentable}
\newcommand{\rrep}{real-rep\-re\-sentable}
\newcommand{\brep}{Brumfiel-rep\-re\-sentable}
\newcommand{\jirr}{join-ir\-re\-duc\-i\-ble}

\newcommand{\eqdef}{\underset{\mathrm{def}}{=}}

\newcommand{\lextimes}{\times_{\mathrm{lex}}}
\newcommand{\Pow}{\mathfrak{P}}
\DeclareMathOperator{\Cr}{C_r}

\DeclareMathOperator{\Spec}{Spec}
\DeclareMathOperator{\Specl}{Spec_{\ell}}
\DeclareMathOperator{\Specr}{Spec_{r}}
\DeclareMathOperator{\Specb}{Spec_{B}}

\newcommand{\FV}{\operatorname{F}_{\cV}}

\DeclareMathOperator{\tcl}{cl}

\DeclareMathOperator{\card}{card}

\newcommand{\les}{\leqslant}

\newcommand{\ga}{\alpha}
\newcommand{\gb}{\beta}

\newcommand{\gf}{\varphi}
\newcommand{\gk}{\kappa}
\newcommand{\gl}{\lambda}

\newcommand{\go}{\omega}
\newcommand{\eps}{\varepsilon}

\newcommand{\bgk}{\boldsymbol{\kappa}}

\newcommand{\beps}{\boldsymbol{\varepsilon}}

\newcommand{\bck}[1]{[\![{#1}]\!]}

\newcommand{\gL}{\Lambda}
\newcommand{\gO}{\Omega}

\newcommand{\sd}{\mathbin{\smallsetminus}}

\newcommand{\two}{\mathbf{2}}
\newcommand{\three}{\mathbf{3}}
\newcommand{\four}{\mathbf{4}}

\newcommand{\ol}[1]{\overline{#1}}

\newcommand{\pI}[1]{\bigl({#1}\bigr)}
\newcommand{\pII}[1]{\Bigl({#1}\Bigr)}

\newcommand{\set}[1]{\left\{#1\right\}}
\newcommand{\setm}[2]{\set{{#1}\mid{#2}}}
\newcommand{\vecm}[2]{\left({#1}\mid{#2}\right)}
\newcommand{\seq}[1]{\langle{#1}\rangle}
\newcommand{\seql}[1]{{\langle{#1}\rangle}^{\ell}}
\newcommand{\seqr}[1]{{\langle{#1}\rangle}^{\mathrm{r}}}

\newcommand{\op}{\mathrm{op}}

\DeclareMathOperator{\rF}{F}

\newcommand{\es}{\varnothing}
\newcommand{\res}{\mathbin{\restriction}}

\newcommand{\ZZ}{\mathbb{Z}}

\newcommand{\RR}{\mathbb{R}}

\newcommand{\Idlc}{\operatorname{Id}^{\ell}_{\mathrm{c}}}

\DeclareMathOperator{\Idrc}{Id^r_c}

\DeclareMathOperator{\Cond}{Cond}

\newcommand{\cC}{{\mathcal{C}}}

\newcommand{\cKo}{{\overset{\circ}{\mathcal{K}}}}

\newcommand{\cL}{{\mathcal{L}}}

\newcommand{\cV}{{\mathcal{V}}}

\numberwithin{equation}{section}

\newtheorem*{stat}{\name}
\newcommand{\name}{testing}

\theoremstyle{plain}

\newtheorem{theorem}{Theorem}[section]
\newtheorem{proposition}[theorem]{Proposition}
\newtheorem{corollary}[theorem]{Corollary}
\newtheorem{lemma}[theorem]{Lemma}

\theoremstyle{definition}

\newtheorem{definition}[theorem]{Definition}
\newtheorem{notation}[theorem]{Notation}

\theoremstyle{remark}
\newtheorem{remark}[theorem]{Remark}

\newcommand{\qedc}{{\qed}~{\rm Claim~{\theclaim}.}}
\newcommand{\qedsc}{{\qed}~{\rm Claim.}}

\numberwithin{figure}{section}
\numberwithin{table}{section}

\newcommand{\ba}{\boldsymbol{a}}
\newcommand{\bb}{\boldsymbol{b}}
\newcommand{\bc}{\boldsymbol{c}}

\newcommand{\be}{\boldsymbol{e}}
\newcommand{\bbf}{\boldsymbol{f}}
\newcommand{\bg}{\boldsymbol{g}}
\newcommand{\bh}{\boldsymbol{h}}

\newcommand{\br}{\boldsymbol{r}}

\newcommand{\bu}{\boldsymbol{u}}

\newcommand{\bx}{\boldsymbol{x}}
\newcommand{\by}{\boldsymbol{y}}

\newcommand{\xCN}{\mathbf{CN}}
\newcommand{\xS}{\mathbf{S}}
\newcommand{\xell}{\boldsymbol{\ell}}
\newcommand{\xSell}{\xS\xell}
\newcommand{\xBr}{\mathbf{Br}}
\newcommand{\xSBr}{\xS\xBr}
\newcommand{\xR}{\mathbf{R}}
\newcommand{\xSR}{\xS\xR}

\newcommand{\bD}{\boldsymbol{D}}

\newcommand{\vx}{\mathsf{x}}
\newcommand{\vy}{\mathsf{y}}
\newcommand{\vt}{\mathsf{t}}

\newcommand{\scL}{\mathbin{\mathscr{L}}}

\title[Real spectrum, $\ell$-spectrum, Brumfiel spectrum]%
{Real spectrum versus $\ell$-spectrum\\
via Brumfiel spectrum}

\author[F. Wehrung]{Friedrich Wehrung}
\address{LMNO, CNRS UMR 6139\\
D\'epartement de Math\'ematiques\\
Universit\'e de Caen Normandie\\
14032 Caen cedex\\
France}
\email{friedrich.wehrung01@unicaen.fr}
\urladdr{https://wehrungf.users.lmno.cnrs.fr}

\date{\today}


\begin{document}

\begin{abstract}
It is well known that the real spectrum of any commutative unital ring, and the $\ell$-spectrum of any Abelian lat\-tice-or\-dered group with order-unit, are all completely normal spectral spaces.
We prove the following results:
\begin{enumerater}
\item\label{RS2lS}
Every real spectrum can be embedded, as a spectral subspace, into some $\ell$-spectrum.

\item\label{RSnotlS}
Not every real spectrum is an $\ell$-spectrum.

\item\label{subRS}
A spectral subspace of a real spectrum may not be a real spectrum.

\item\label{lSnotRS}
Not every $\ell$-spectrum can be embedded, as a spectral subspace, into a real spectrum.

\item\label{CNnotlS}
There exists a completely normal spectral space which cannot be embedded, as a spectral subspace, into any $\ell$-spectrum.

\end{enumerater}
The commutative unital rings and Abelian lat\-tice-or\-dered groups in~\eqref{RSnotlS}, \eqref{subRS}, \eqref{lSnotRS} all have cardinality~$\aleph_1$\,, while the spectral space of~\eqref{CNnotlS} has a basis of cardinality~$\aleph_2$\,.
Moreover, \eqref{subRS} solves a problem by Mellor and Tressl.
\end{abstract}

\maketitle

\section{Introduction}\label{S:Intro}
Denote by~$\xS\mathbf{X}$ the class of all spectral subspaces of members of a class~$\mathbf{X}$ of spectral spaces.
Most of the paper is devoted to proving the containments and non-containments, between classes of spectral spaces, represented in Figure~\ref{Fig:Spectra}.
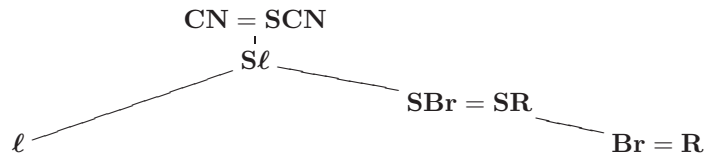
\begin{figure}[htb]
 \[
 \xymatrixrowsep{.22pc}\xymatrixcolsep{2pc}
 \def\labelstyle{\displaystyle}
 \xymatrix{
 && \xCN=\xS\xCN &&\\
 && \xSell\ar@{-}[u] &&\\
 &&& \xSBr=\xSR\ar@{-}[ul] &\\
 \xell\ar@{-}[uurr] &&&& \xBr=\xR\ar@{-}[ul]
 }
 \]
\caption{Classes of completely normal spectral spaces}
\label{Fig:Spectra}
\end{figure}
The classes in question are the following:
\begin{itemize}
\item
$\xCN$, the class of all completely normal spectral spaces;

\item
$\xell$, the class of $\ell$-spectra of all Abelian \lgrp{s} with order-unit;

\item
$\xBr$, the class of Brumfiel spectra of all commutative unital \fring{s};

\item
$\xR$, the class of real spectra of all commutative unital rings.

\end{itemize}

The context of our work is the following.
The classical construction of the Zariski spectrum of a commutative unital ring (cf. Subsection~\ref{Su:Zspec}) extends to many contexts, including distributive lattices, \emph{lat\-tice-or\-dered groups} (\emph{\lgrp{s}} for short), \emph{partially or\-dered rings}, yielding \emph{Stone duality}, the \emph{$\ell$-spectrum}, and the \emph{real spectrum}, respectively.
All the topological spaces thus obtained are \emph{spectral spaces}, that is, sober spaces in which the compact open subsets are a basis of the topology, closed under finite intersection.
Conversely, every spectral space is the spectrum of some bounded distributive lattice (Stone~\cite{Stone38a}) and also of some commutative unital ring (Hochster~\cite{Hoch1969}).

The paper will focus on the \emph{$\ell$-spectrum} of an Abelian \lgrp\ (cf. Subsection~\ref{Su:lspec}) and the \emph{real spectrum} of a commutative unital ring (cf. Subsection~\ref{Su:Rspec}).
Those two frameworks are connected by the \emph{Brumfiel spectrum} of a commutative \fring\ (cf. Subsection~\ref{Su:Bspec}).
All the spectral spaces thus obtained are \emph{completely normal}, that is, for all elements~$x$ and~$y$ in the closure of a singleton~$\set{z}$, either~$x$ belongs to the closure of~$\set{y}$ or~$y$ belongs to the closure of~$\set{x}$.

Prior to the present paper, part of the picture (Figure~\ref{Fig:Spectra}) was already known:
\begin{itemize}
\item
Delzell and Madden~\cite{DelMad1994} proved that $\xell\subsetneqq\xCN$ and $\xR\subsetneqq\xCN$.

\item
Delzell and Madden's result got amplified in Mellor and Tressl~\cite{MelTre2012}, who established that \emph{any class of spectral spaces containing~$\xR$, whose Stone dual lattices are definable by a class of~$\scL_{\infty,\gl}$-formulas for some infinite cardinal~$\gl$, has a member outside~$\xSR$}.
In particular, the class of all Stone duals of the spaces from~$\xR$ (resp., $\xSR$) are not $\scL_{\infty,\gl}$-definable.
Further, $\xSR\subsetneqq\xCN$.

\item
Delzell and Madden~\cite[Proposition~3.3]{DelMad1995} observed that $\xR\subseteq\xBr$.

\item
It follows easily from Madden and Schwartz \cite{MadSch1999} and Schwartz~\cite{Sch1989} that $\xBr\subseteq\xR$.
Consequently, $\xBr=\xR$ (cf. Corollary~\ref{C:BrS2RS1}).

\item
The author proved in~\cite{MV1} that \emph{every second countable completely normal spectral space is in~$\xell$}, and that moreover, \emph{the class of all Stone duals of spaces from~$\xell$ is not~$\scL_{\infty,\go}$-definable}.

\item
The author provided an example in~\cite[\S~5]{MV1} showing that $\xell\subsetneqq\xSell$.

\end{itemize}

The missing pieces provided in the present paper are the following:

\begin{itemize}
\item
Every Brumfiel spectrum, thus also every real spectrum, can be embedded, as a spectral subspace, into some $\ell$-spectrum.
This is stated in Corollary~\ref{C:Idl2Idr}.
Hence, $\xSBr\subseteq\xSell$.

\item
Not every real spectrum is an $\ell$-spectrum.
This is established in Theorem~\ref{T:RCRnotlrep}, \emph{via} the construction of a \emph{condensate}.
Hence, $\xR\not\subseteq\xell$.

\item
A spectral subspace of a real spectrum may not be a Brumfiel spectrum (thus also not a real spectrum).
This is stated in Corollary~\ref{C:NonHomrealSpec}, \emph{via} the construction of a condensate.
It follows that $\xSR\not\subseteq\xBr$.
This solves a problem of Mellor and Tressl~\cite{MelTre2012}.

\item
Not every $\ell$-spectrum can be embedded, as a spectral subspace, into a Brumfiel spectrum (thus also not into a real spectrum).
Hence, $\xell\not\subseteq\xSBr$.
This is stated in Corollary~\ref{C:NonrReprlGrp}.

\item
There exists a completely normal spectral space which cannot be embedded, as a spectral subspace, into any $\ell$-spectrum.
This is stated in Corollary~\ref{C:FVInotlrep}.
Hence, $\xSell\subsetneqq\xCN$.
The spectral space constructed there has~$\aleph_2$ compact open members.
The proof begins by coining a class of infinitary statements satisfied by all homomorphic images of \lrep\ lattices (Lemma~\ref{L:RefI}).
Then the proof technique loosely follows the one introduced by the author in~\cite{NonMeas} for solving representation problems of congruence lattices of lattices and nonstable K-theory of von~Neumann regular rings.
\end{itemize}

We point out that although not every $\ell$-spectrum is a real spectrum (cf. Corollary~\ref{C:NonrReprlGrp}), there is a formally related problem, on the interaction between Abelian \lgrp{s} and commutative rings, with a well known positive solution.
Let~$D$ be an integral domain with group of units~$U$ and field of fractions~$K$.
Denote by~$K^{\times}$ the multiplicative group of all nonzero elements of~$K$.
The \emph{group of divisibility of~$D$} (cf. Mo\v{c}ko\v{r}~\cite{Moc1983}) is the quotient group~$K^{\times}/U$, endowed with the unique translation-invariant partial ordering with positive cone~$D/U$.
Every Abelian \lgrp\ is the group of divisibility of some integral domain, which, in addition, can be taken a \emph{Bezout domain} (cf. Anderson \cite[page~4]{And1989}, where the result is credited to Krull, Jaffard, Kaplansky, and Ohm).
Thus, Corollary~\ref{C:NonrReprlGrp} illustrates the gap between the group of divisibility and the real spectrum.

Since Stone duality is more conveniently stated with \emph{bounded} distributive lattices, our results on spectral spaces are mostly formulated for commutative, \emph{unital} rings and Abelian \lgrp{s} \emph{with order-unit}.
On the other hand, most of our lattice-theoretical results are valid for lattices without top element, and thus formulated in that more general context.

\section{Basic concepts}\label{S:Basic}

For \lgrp{s} and \fring{s}, we refer the reader to Bigard, Keimel, and Wolfenstein~\cite{BKW} or Anderson and Feil~\cite{AnFe}.
For any partially ordered Abelian group~$G$, we set $G^+\eqdef\setm{x\in G}{x\geq0}$ (the \emph{positive cone} of~$G$) and
$G^{++}\eqdef\setm{x\in G}{x>0}$.
For $a,b\in G^+$, let $a\ll b$ hold if $ka\leq b$ for every positive integer~$k$.

For partially ordered Abelian groups~$G$ and~$H$, the \emph{lexicographical product} of~$G$ by~$H$, denoted~$G\lextimes H$, is the product group~$G\times H$, endowed with the positive cone consisting of all pairs $(x,y)$ with either $x>0$ or ($x=0$ and $y\geq0$).

For any chain~$\gL$, we denote by~$\ZZ\seq{\gL}$ the lexicographical power, of the chain~$\ZZ$ of all integers, by~$\gL$.
Hence the elements of~$\ZZ\seq{\gL}$ have the form $x=\sum_{i=1}^nk_ic_{\xi_i}$, where each $k_i\in\ZZ\setminus\set{0}$ and $\xi_1<\cdots<\xi_n$ in~$\gL$, and~$x$ belongs to the positive cone of~$\ZZ\seq{\gL}$ if{f} either $n=0$ (i.e., $x=0$) or $k_n>0$.
This endows~$\ZZ\seq{\gL}$ with a structure of a totally ordered Abelian group.

A \emph{lat\-tice-or\-dered group}, or \emph{\lgrp} for short, is a group endowed with a trans\-la\-tion-in\-vari\-ant lattice ordering.
All our \lgrp{s} will be Abelian and will thus be denoted additively.
Elements~$x$ and~$y$, in an \lgrp, are \emph{orthogonal} if $x\wedge y=0$.

A subset~$I$ in an Abelian \lgrp~$G$ is an \emph{$\ell$-ideal} if it simultaneously a subgroup of~$G$ and an order-convex sublattice of~$G$.

For any elements~$a$ and~$b$ in an Abelian \lgrp~$G$, we will set $a^+\eqdef a\vee0$, $a^-\eqdef(-a)\vee0$, $|a|\eqdef a\vee(-a)$, and $a\sd b\eqdef(a-b)^+=a-(a\wedge b)$.

A \emph{lat\-tice-or\-dered ring} is a ring endowed with a lattice ordering invariant under additive translations and preserved by multiplicative translations by positive elements.
A lat\-tice-or\-dered ring~$A$ is an \emph{\fring} if $x\wedge y=0$ implies that $x\wedge yz=x\wedge zy=0$ whenever $x,y,z\in A^+$ and $x\wedge y=0$.
Equivalently, $A$ is a subdirect product of totally ordered (not necessarily unital) rings (cf. Bigard, Keimel, and Wolfenstein \cite[Th\'eor\`eme~9.1.2]{BKW}).

\begin{lemma}[folklore]\label{L:Gabcsd}
Let~$G$ be an Abelian \lgrp\ and let $a,b,c\in G$.
Then
\begin{itemize}
\item
$a\sd c\leq(a\sd b)+(b\sd c)$.

\item
$(a\sd b)\wedge(b\sd a)=0$.

\item
If, in addition, $G$ is the underlying additive \lgrp\ of an \fring~$A$ and $c\in A^+$, then
$ca\sd cb=c(a\sd b)$ and $ac\sd bc=(a\sd b)c$.
\end{itemize}
\end{lemma}

A subset~$I$ in an \fring~$A$ is an \emph{$\ell$-ideal} if it is, simultaneously, an ideal of the underlying ring of~$A$ and an order-convex sublattice of~$A$.

Totally ordered rings are particular cases of \fring{s}.
About those, we will need the following lemma.

\begin{lemma}\label{L:1notinP}
Let~$A$ be a totally ordered \pup{not necessarily unital} commutative domain and let~$I$ be a proper order-convex ideal of~$A$.
Then for every $x\in I$ and every $a\in A$, the relation $|xa|\ll|a|$ holds.
\end{lemma}

\begin{proof}
We will use repeatedly the fact that for every $c\in A^{++}$, the assignment $t\mapsto tc$ defines an order-embedding of~$A$ into itself.
Since~$A$ is totally ordered, we may assume that $a\geq0$ and $x\geq0$.
Let $n<\go$ and suppose that $nxa>a$ (so $a>0$).
Then for every $b\in A^+$, $nxab\geq ab\geq0$, thus (as $a>0$) $nxb\geq b\geq0$.
Since~$x\in I$ and~$I$ is an ideal of~$A$, we get $nxb\in I$.
Since~$I$ is order-convex, it follows that~$b\in I$.
This holds for every $b\in A^+$, whence $I=A$, a contradiction.
Since~$A$ is totally ordered, it follows that $nxa\leq a$.
\end{proof}

For lattice theory we refer the reader to Gr\"atzer~\cite{LTF}, Johnstone~\cite{Johnst1982}.
For any elements~$a$ and~$b$ in a distributive lattice~$D$ with zero, a \emph{splitting} of~$(a,b)$ is a pair~$(x,y)$ of elements of~$D$ such that $a\vee b=a\vee y=x\vee b$ and $x\wedge y=0$.
Observe that in that case, $x\leq a$ and $y\leq b$.
We say that~$D$ is \emph{completely normal} if every pair of elements in~$D$ has a splitting.

We denote by~$P^{\op}$ the \emph{opposite poset} of a poset~$
P$.
For any functions~$f$ and~$g$ with common domain~$X$, we set
 \[
 \bck{f\neq g}\eqdef\setm{x\in X}{f(x)\neq g(x)}\,.
 \]
We denote by~$\Pow(X)$ the powerset of any set~$X$, ordered under set inclusion.
We denote by~$\go_\ga$, or~$\aleph_\ga$ according to the context (``ordinal versus cardinal''), the~$\ga$th infinite cardinal, and we set $\go\eqdef\go_0=\set{0,1,2,\dots}$.

Throughout the paper, ``countable'' means ``at most countable''.

\section{Stone duality between distributive lattices with zero and generalized spectral spaces}
\label{S:Spectral}

In this section we recall a few well known facts on Stone duality for bounded distributive lattices.
For references and more details, see Johnstone \cite[\S~II.3]{Johnst1982}, Gr\"atzer \cite[\S~II.5]{LTF}.

\begin{definition}\label{D:Spectral}
For a topological space~$X$, we denote by~$\cKo(X)$ the set of all compact%
\footnote{Throughout the paper, ``compact'' means what some other references call ``quasicompact''; in particular, it does not imply Hausdorff.}
open subsets of~$X$, ordered under set inclusion.
We say that~$X$ is
\begin{itemize}
\item[---]
\emph{sober}, if every \jirr\ member, of the lattice of all closed subsets of~$X$, is the closure of a unique singleton%
\footnote{
Due to the uniqueness, every sober space is~$\mathrm{T}_0$ (not all references assume this).
}
;

\item[---]
\emph{generalized spectral}, if it is sober, $\cKo(X)$ is a basis of the topology of~$X$, and $U\cap V$ is compact whenever~$U$ and~$V$ are compact open subsets of~$X$;

\item[---]
\emph{spectral}, if it is simultaneously compact and generalized spectral. 
\end{itemize}

The \emph{specialization preorder} on~$X$ is defined by
 \[
 x\les y\text{ if }y\in\tcl_{X}(\set{x})\,,\quad\text{for all }x,y\in X\,.
 \]
\end{definition}

The \emph{spectrum}~$\Spec{D}$, of a distributive lattice~$D$ with zero, is defined as the set of all (proper) prime ideals of~$D$, endowed with the closed sets $\setm{P\in\Spec{D}}{I\subseteq P}$, for subsets (equivalently, ideals)~$I$ of~$D$.
The specialization order on~$\Spec{D}$ is just set-theoretical inclusion.
The correspondence between distributive lattices with zero and generalized spectral spaces is spelled out in the following result, originating in Stone~\cite{Stone38a}.

\goodbreak
\begin{theorem}[Stone]\label{T:Stone}
\hfill
\begin{itemize}
\item
For every distributive lattice~$D$ with zero, the space~$\Spec{D}$ is generalized spectral and the assignment $a\mapsto\setm{P\in\Spec{D}}{a\notin P}$ defines an isomorphism $\ga_D\colon D\to\cKo(\Spec{D})$.

\item
For every generalized spectral space~$X$, $\cKo(X)$ is a distributive lattice with zero and the assignment $x\mapsto\setm{U\in\cKo(X)}{x\notin U}$ defines a homeomorphism $\xi_X\colon X\to\Spec{\cKo(X)}$.
\end{itemize}
\end{theorem}

For distributive lattices~$D$ and~$E$, a $0$-lattice homomorphism $f\colon D\to E$, and $Q\in\Spec{E}$, the inverse image $f^{-1}[Q]$ may be the whole of~$D$, in which case it does not belong to~$\Spec{D}$ (prime ideals are assumed to be proper).
This does not happen if we assume the map~$f$ to be \emph{cofinal}, that is, every element of~$E$ is bounded above by some element of the range of~$f$.

Say that a map $\gf\colon X\to Y$, between topological spaces, is \emph{spectral} if the inverse image under~$\gf$, of any compact open subset of~$Y$, is a compact open subset of~$X$.
If~$X$ and~$Y$ are both generalized spectral, then the map $\cKo(\gf)\colon\cKo(Y)\to\cKo(X)$, $V\mapsto\gf^{-1}[V]$ is a cofinal $0$-lattice homomorphism.
Hence we obtain the following statement of Stone's duality (spelled out in Rump and Yang \cite[page~63]{RumYan2009}), which extends the classical Stone duality between bounded distributive lattices and spectral spaces.

\begin{theorem}[Stone]\label{T:StoneDual}
The category of all distributive lattices with zero, with cofinal $0$-lattice homomorphisms, and the category of all generalized spectral spaces, with spectral maps, are dual, with respect to the natural transformations~$\ga$ and~$\xi$ given in Theorem~\textup{\ref{T:Stone}} and the functors given as follows:
\begin{itemize}
\item
The dual of a distributive lattice~$D$ with zero is its spectrum~$\Spec{D}$.
The dual of a $0$-lattice homomorphism $f\colon D\to E$ is the map\newline $\Spec{f}\colon\Spec{E}\to\Spec{D}$, $Q\mapsto f^{-1}[Q]$.

\item
The dual of a generalized spectral space~$X$ is the lattice~$\cKo(X)$.
The dual of a spectral map $\gf\colon X\to Y$ is the map $\cKo(\gf)\colon\cKo(Y)\to\cKo(X)$, $V\mapsto\gf^{-1}[V]$.

\end{itemize}
\end{theorem}

\begin{remark}\label{Rk:StoneDual}
The case where~$\gf$ is the inclusion map from a generalized spectral space~$X$ into a generalized spectral space~$Y$ is interesting.
We say that~$X$ is a \emph{spectral subspace} of~$Y$ if the topology of~$X$ is the topology induced by the topology of~$Y$ and the inclusion map from~$X$ into~$Y$ is spectral.
In that case, the dual map $\cKo(Y)\to\cKo(X)$, $V\mapsto X\cap V$ is a surjective lattice homomorphism.
Conversely, for every surjective lattice homomorphism $f\colon D\twoheadrightarrow E$, the spectral map $\Spec{f}\colon\Spec{E}\to\Spec{D}$ is a \emph{spectral embedding}, that is, it embeds~$\Spec{E}$ into~$\Spec{D}$ as a spectral subspace.
Hence, \emph{spectral subspaces correspond, \emph{via} Stone duality, to surjective lattice homomorphisms}.
\end{remark}

The generalized spectral spaces~$X$ that we will consider in this paper will mostly be \emph{completely normal}.
By Monteiro~\cite[Th\'e\-o\-r\`e\-me~V.3.1]{Mont1954}, this is equivalent to saying that the dual lattice~$\cKo(X)$ is completely normal (cf. Section~\ref{S:Basic} for the definition of completely normal lattices).

\section{Zariski, $\ell$, Brumfiel, real: spectra and lattices}
\label{S:ZlBR}

In this section we recall some well known facts on the various sorts of spectra and distributive lattices that will intervene in the paper.
We also include a few new results, such as Lemma~\ref{L:BrumConvex}.
For more details and references, we refer the reader to Delzell and Madden~\cite{DelMad1995}, Johnstone \cite[Chapter~5]{Johnst1982}, Keimel~\cite{Keim1995}, Coste and Roy~\cite{CosRoy1982}, Dickmann~\cite[Chapter~6]{Dickm1985}.

\subsection{Zariski spectrum}\label{Su:Zspec}
The (\emph{Zariski}) \emph{spectrum} of a commutative unital ring~$A$ is defined as the set~$\Spec{A}$ of all prime ideals of~$A$, endowed with the topology whose closed sets are exactly the sets $\Spec(A,I)\eqdef\setm{P\in\Spec{A}}{I\subseteq P}$, for subsets (equivalently, radical ideals)~$I$ of~$A$.

Denote by $\seqr{a_1,\dots,a_m}$ the radical ideal of~$A$ generated by elements $a_1$, \dots, $a_m$ of~$A$, and denote%
\footnote{The subscript~``c'' stands for ``compact'', which is the lat\-tice-theoretical formalization of ``finitely generated''.
The superscript~``r'' stands for ``radical''.}
by~$\Idrc{A}$ the set of all ideals of~$A$ of the form $\seqr{a_1,\dots,a_m}$ (\emph{finitely generated radical ideals}), ordered by set inclusion.
Due to the formulas
 \begin{align}
 \seqr{a_1,\dots,a_m}\vee\seqr{b_1,\dots,b_n}&=
 \seqr{a_1,\dots,a_m,b_1,\dots,b_n}\,,
 \label{Eq:JoinFinGen}\\
 \seqr{a_1,\dots,a_m}\cap\seqr{b_1,\dots,b_n}&=
 \seqr{a_ib_j\mid 1\leq i\leq m\text{ and }1\leq j\leq n}\,,
 \label{Eq:IntersFinGen}
 \end{align}
(where~$\vee$ stands for the join in the lattice of all radical ideals of~$A$),
$\Idrc{A}$ is a $0$-sublattice of the distributive lattice of all radical ideals of~$A$.

Since every radical ideal of~$A$ is the intersection of all prime ideals containing it, $\Idrc{A}$ is the Stone dual of~$\Spec{A}$ (cf. Delzell and Madden \cite[page~115]{DelMad1995}):

\begin{proposition}\label{P:IdrcA}
The Zariski spectrum~$\Spec{A}$, of a commutative unital ring~$A$, is a spectral space, and the assignment $I\mapsto\setm{P\in\Spec{A}}{I\not\subseteq P}$ defines an isomorphism from~$\Idrc{A}$ onto the Stone dual~$\cKo(\Spec{A})$ of~$\Spec{A}$.
\end{proposition}

Due to the following deep result by Hochster~\cite{Hoch1969}, there is no need to give a name to the class of all lattices of the form~$\Idrc{A}$.

\begin{theorem}[Hochster]\label{T:Hochster}
Every spectral space is homeomorphic to the Zariski spectrum of some commutative unital ring.
Hence, every bounded distributive lattice is isomorphic to~$\Idrc{A}$ for some commutative unital ring~$A$.
\end{theorem}

\subsection{$\ell$-spectrum and \lrep\ lattices}\label{Su:lspec}
The \emph{$\ell$-spectrum} of an Abelian \lgrp~$G$ is defined as the set $\Specl{G}$ of all prime $\ell$-ideals of~$G$, endowed with the topology whose closed sets are exactly the $\setm{P\in\Specl{G}}{I\subseteq P}$, for subsets (equivalently, $\ell$-ideals)~$I$ of~$G$.

Denote by $\seql{a_1,\dots,a_m}$, or $\seql{a_1,\dots,a_m}_G$ if~$G$ needs to be specified, the $\ell$-ideal of~$G$ generated by elements $a_1$, \dots, $a_m$ of~$G$, and denote
by~$\Idlc{G}$ the set of all $\ell$-ideals of~$G$ of the form $\seqr{a_1,\dots,a_m}$ (\emph{finitely generated $\ell$-ideals}), ordered by set inclusion.
Since $\seql{a_1,\dots,a_m}=\seql{a}$ where $a\eqdef\sum_{i=1}^m|a_i|$, we get $\Idlc{G}=\setm{\seql{a}}{a\in G^+}$.
Due to the formulas
 \begin{equation}\label{Eq:jjmmseql}
 \seql{a}\vee\seql{b}=\seql{a+b}\text{ and }
 \seql{a}\cap\seql{b}=\seql{a\wedge b}\,,\quad
 \text{for all }a,b\in G^+
 \end{equation}
(where~$\vee$ stands for the join in the lattice of all $\ell$-ideals of~$A$),
$\Idlc{G}$ is a $0$-sublattice of the distributive lattice of all $\ell$-ideals of~$G$.
It has a top element if{f}~$G$ has an order-unit.

Since every $\ell$-ideal of~$G$ is the intersection of all prime $\ell$-ideals containing it, $\Idrc{G}$ is the Stone dual of~$\Specl{G}$:

\begin{proposition}\label{P:IdlcG}
The $\ell$-spectrum~$\Specl{G}$, of any Abelian \lgrp~$G$, is a generalized spectral space, and the assignment $I\mapsto\setm{P\in\Spec{G}}{I\not\subseteq P}$ defines an isomorphism from~$\Idlc{G}$ onto the Stone dual~$\cKo(\Specl{G})$ of~$\Specl{G}$.
\end{proposition}

Following terminology from Iberkleid, Mart{\'{\i}}nez, and McGovern \cite{IMM2011} and Weh\-rung~\cite{MV1}, we recall the following definition.

\begin{definition}\label{D:closedmap}
For distributive lattices~$A$ and~$B$, a map $f\colon A\to B$ is \emph{closed} if for all $a_0,a_1\in A$ and all $b\in B$, if $f(a_0)\leq f(a_1)\vee b$, then there exists $x\in A$ such that $a_0\leq a_1\vee x$ and $f(x)\leq b$.
\end{definition}

The following lemma is established in Wehrung~\cite[\S~3]{MV1}.

\begin{lemma}\label{L:Idcf}
Let~$A$ and~$B$ be Abelian \lgrp{s} and let $f\colon A\to B$ be an \lhom.
Then the map $\Idlc{f}\colon\Idlc{A}\to\Idlc{B}$, $\seql{x}\mapsto\seql{f(x)}$ is a closed $0$-lattice homomorphism.
\end{lemma}

In particular, the assignments $G\mapsto\Idlc{G}$, $f\mapsto\Idlc{f}$ define a functor, from the category of all Abelian \lgrp{s} with \lhom{s}, to the category of all distributive lattices with zero with closed $0$-lattice homomorphisms.
It is well known that this functor preserves nonempty finite direct products and directed colimits.

Say that a lattice~$D$ is \emph{\lrep} if it is isomorphic to~$\Idlc{G}$ for some Abelian \lgrp~$G$.
Equivalently, the spectrum of~$D$ is homeomorphic to the $\ell$-spectrum of some Abelian \lgrp.
This terminology is extended to diagrams~$\vec{D}$ of distributive lattices with zero and $0$-lattice homomorphisms, by saying that $\vec{D}\cong\Idlc\vec{G}$ for some diagram~$\vec{G}$ of Abelian \lgrp{s} and \lhom{s}.

It is well known that every \lrep\ lattice is completely normal.
The author established the following result in~\cite{MV1}.

\begin{theorem}\label{T:MV1}
Every countable completely normal distributive lattice with zero is \lrep.
On the other hand, the class of all \lrep\ lattices cannot be defined by a class of~$\scL_{\infty,\go}$-formulas of lattice theory.
\end{theorem}

The following easy lemma is established in Wehrung~\cite[\S~3]{MV1}.

\begin{lemma}\label{L:FactClosed}
Let~$G$ be an Abelian \lgrp, let~$S$ be a distributive lattice with zero, and let $\gf\colon\Idlc G\twoheadrightarrow S$ be a closed surjective $0$-lattice homomorphism.
Then $I\eqdef\setm{x\in G}{\gf(\seql{x})=0}$ is an $\ell$-ide\-al of~$G$, and there is a unique isomorphism $\psi\colon\Idlc(G/I)\to S$ such that $\psi(\seql{x+I})=\gf(\seql{x})$ for every $x\in G^+$.
\end{lemma}

\subsection{Brumfiel spectrum and \brep\ lattices}\label{Su:Bspec}
For any commutative \fring~$A$, we say that a (proper) $\ell$-ideal~$P$ is \emph{prime} if it is both an $\ell$-ideal and prime as a ring ideal.
Then~$P$ is also prime as an $\ell$-ideal of the underlying additive \lgrp\ of~$A$, that is, $x\wedge y\in P$ implies that either $x\in P$ or $y\in P$, whenever $x,y\in A$.

The \emph{Brumfiel spectrum} of a commutative \fring~$A$ is defined as the set~$\Specb{A}$ of all prime $\ell$-ideals of~$A$, endowed with the topology whose closed sets are exactly the $\setm{P\in\Specb{A}}{I\subseteq P}$, for subsets (equivalently, radical $\ell$-ideals)~$I$ of~$A$.

Denote by $\seqr{a_1,\dots,a_m}$, or $\seqr{a_1,\dots,a_m}_A$ if~$A$ needs to be specified, the radical $\ell$-ideal of~$A$ generated by elements $a_1$, \dots, $a_m$ of~$A$, and denote by~$\Idrc{A}$ the set of all ideals of~$A$ of the form $\seqr{a_1,\dots,a_m}$ (\emph{finitely generated radical ideals}),%
\footnote{Although we are using, for radical $\ell$-ideals, the same notation as the one in Subsection~\ref{Su:Zspec} for radical ideals, the context will always make it clear which concept is used.}
ordered by set inclusion.
Since $\seqr{a_1,\dots,a_m}=\seql{a}$ where $a\eqdef\sum_{i=1}^m|a_i|$, we get $\Idrc{A}=\setm{\seqr{a}}{a\in A^+}$.
Due to the formulas
 \begin{equation}\label{Eq:jjmmseqr}
 \seqr{a}\vee\seqr{b}=\seqr{|a|+|b|}\text{ and }
 \seqr{a}\cap\seqr{b}=\seqr{|a|\wedge|b|}=\seqr{ab}\,,\quad
 \text{for all }a,b\in A
 \end{equation}
(where~$\vee$ stands for the join in the lattice of all radical $\ell$-ideals of~$A$),
$\Idrc{A}$ is a $0$-sublattice of the distributive lattice of all radical $\ell$-ideals of~$A$.
If~$A$ is unital, then~$\Idrc{A}$ has a top element.

Since every radical $\ell$-ideal of~$A$ is the intersection of all prime $\ell$-ideals containing it, $\Idrc{A}$ is the Stone dual of~$\Specb{A}$ (cf. Delzell and Madden \cite[Proposition~4.2]{DelMad1995}):

\begin{proposition}\label{P:IdrcABr}
The Brumfiel spectrum~$\Specb{A}$, of a commutative \fring~$A$, is a generalized spectral space, and the assignment $I\mapsto\setm{P\in\Specb{A}}{I\not\subseteq P}$ defines an isomorphism from~$\Idrc{A}$ onto the Stone dual~$\cKo(\Specb{A})$ of~$\Specb{A}$.
\end{proposition}

Say that a lattice~$D$ is \emph{\brep} if it is isomorphic to~$\Idrc{A}$ for some commutative \fring~$A$.
Equivalently, the spectrum of~$D$ is homeomorphic to the Brumfiel spectrum of some commutative \fring.
As in Subsection~\ref{Su:lspec}, this terminology is extended to diagrams of lattices, in a standard fashion.

It is well known that every \brep\ lattice is completely normal.
We will see, with Corollary~\ref{C:NonrReprlGrp} in the present paper, that not every \lrep\ lattice (thus, \emph{a fortiori}, not every completely normal distributive lattice) is \brep.

The following result is an analogue of Lemma~\ref{L:FactClosed} for \fring{s}.
Its proof is similar and we omit it.

\begin{lemma}\label{L:FactClosedRng}
Let~$A$ be a commutative \fring, let~$S$ be a distributive lattice with zero, and let $\gf\colon\Idrc A\twoheadrightarrow S$ be a closed surjective $0$-lattice homomorphism.
Then $I\eqdef\setm{x\in A}{\gf(\seqr{x})=0}$ is a radical $\ell$-ide\-al of~$A$, and there is a unique isomorphism $\psi\colon\Idrc(A/I)\to S$ such that $\psi(\seqr{x+I})=\gf(\seqr{x})$ for every $x\in A^+$.
\end{lemma}

\begin{definition}\label{D:SpecConvex}
Let~$A$ and~$B$ be distributive lattices with zero.
A $0$-lattice homomorphism $f\colon A\to B$ is \emph{convex} if for all $P\in\Spec{A}$, all $Q_0\in\Spec{B}$, and every proper ideal~$J$ of~$B$, if $Q_0\subseteq J$ and $f^{-1}[Q_0]\subseteq P\subseteq f^{-1}[J]$, then there exists $Q\in\Spec{B}$ such that $Q_0\subseteq Q\subseteq J$ and $P=f^{-1}[Q]$.
\end{definition}

The following result extends to the Brumfiel spectrum functor a result originally established for the real spectrum functor in Korollar~4, pages 133--134 of Knebusch and Scheiderer~\cite{KneSch1989}.
Our proof is a straightforward modification of its analogue for real spectra, stated in the forthcoming monograph Dickmann, Schwartz, and Tressl \cite[Theorem~12.3.12]{DST1}.

\begin{lemma}\label{L:BrumConvex}
Let~$A$ and~$B$ be commutative \fring{s} and let $f\colon A\to B$ be a homomorphism of \fring{s}.
Then the map $\Idrc{f}\colon\Idrc{A}\to\Idrc{B}$ is convex.
\end{lemma}

\begin{proof}
We must prove that for every $P\in\Specb{A}$, all $Q_0\in\Specb{B}$, every proper radical $\ell$-ideal~$J$ of~$B$, if $Q_0\subseteq J$ and $f^{-1}[Q_0]\subseteq P\subseteq f^{-1}[J]$, then there exists $Q\in\Specb{B}$ such that $Q_0\subseteq Q\subseteq J$ and $P=f^{-1}[Q]$.
We may replace~$A$ by~$A/f^{-1}[Q_0]$, $B$ by~$B/Q_0$, $J$ by~$J/Q_0$, and~$f$ by the canonical embedding $A/f^{-1}[Q_0]\to B/Q_0$.
Hence, we may assume that~$A$ is an ordered subring of a totally ordered (not necessarily unital) commutative domain~$B$, $Q_0=\set{0}$, $f$ is the inclusion map from~$A$ into~$B$, $P\in\Specb{A}$, $J$ is a proper radical $\ell$-ideal of~$B$, and $P\subseteq J$.
We must find $Q\in\Specb{B}$ such that $P=Q\cap A$ and $Q\subseteq J$.
We set
 \[
 Q\eqdef\setm{y\in J}{(\exists n<\go)(\exists x\in P)\pI{|y|^n\leq x}}\,.
 \]
We claim that~$Q$ is a prime $\ell$-ideal of~$B$.
It is obvious that~$Q$ is an order-convex $\ell$-subgroup of~$B$.
Now let $y\in Q$ and $b\in B$.
We must prove that $yb\in Q$.
Since~$B$ is totally ordered, we may assume that $y,b\in B^+$.
Since~$J$ is an ideal of~$B$, $yb\in J$.
By assumption, there are $n<\go$ and $x\in P$ such that $y^n\leq x$.
It follows that $y^n\in J$, thus, since~$J$ is a radical ideal of~$B$, $y\in J$, and thus $yb^{n+1}\in J$.
Since~$J$ is a proper $\ell$-ideal of~$B$, it follows, using Lemma~\ref{L:1notinP}, that $(yb)^{n+1}=y^n(yb^{n+1})\leq y^n\leq x$, whence $yb\in Q$.
This completes the proof that~$Q$ is an $\ell$-ideal of~$B$.

Let $x,y\in B$ such that $xy\in Q$, we must prove that $x\in Q$ or $y\in Q$.
Since~$B$ is totally ordered, we may assume that $0\leq x\leq y$.
There are $n<\go$ and $p\in P$ such that $(xy)^n\leq p$.
It follows that $x^{2n}\leq(xy)^n\leq p$, whence $x\in Q$, thus completing the proof that~$Q$ is prime.

Now it is obvious that $Q\subseteq J$ and $P=Q\cap A$.
\end{proof}

\subsection{Real spectrum and \rrep\ lattices}\label{Su:Rspec}
Let~$A$ be a commutative unital ring.
A subset~$C$ of~$A$ is a \emph{cone} if it is both an additive and a multiplicative submonoid of~$A$, containing all squares in~$A$.
A cone~$P$ of~$A$ is \emph{prime} if $A=P\cup(-P)$ and the \emph{support} $P\cap(-P)$ is a prime ideal of~$A$.
For a prime cone~$P$, $-1\notin P$ (otherwise $1\in P\cap(-P)$, thus $P\cap(-P)=A$, a contradiction).
We denote by~$\Specr{A}$ the set of all prime cones of~$A$, endowed with the topology generated by all subsets of the form $\setm{P\in\Specr{A}}{a\notin P}$, for $a\in A$, and we call~$\Specr{A}$ the \emph{real spectrum of~$A$}.

It is not so straightforward to describe directly the Stone dual of~$\Specr{A}$.
However, it is possible to reduce it to the Brumfiel spectrum, as follows.
The \emph{universal \fring}~$\rF(A)$ of~$A$ is a commutative unital \fring.
The first statement in the following result is established in Delzell and Madden~\cite[Proposition~3.3]{DelMad1995}.
The second statement follows by using Proposition~\ref{P:IdrcABr}.

\begin{theorem}[Delzell and Madden]\label{T:RS2BS}
Let~$A$ be a commutative unital ring.
The canonical homomorphism $A\to\rF(A)$ induces a homeomorphism between the real spectrum of~$A$ and the Brumfiel spectrum of~$\rF(A)$.
Hence, the Stone dual of~$\Specr{A}$ is $\Idrc{\rF(A)}$.
\end{theorem}

Say that a lattice is \emph{\rrep} if it is isomorphic to the Stone dual of the real spectrum of some commutative unital ring.
As in Subsection~\ref{Su:lspec}, this terminology is extended to diagrams of lattices, in a standard fashion.
It follows from Theorem~\ref{T:RS2BS} that every \rrep\ lattice is isomorphic to~$\Idrc{A}$ for some commutative unital \fring~$A$ (thus it is \brep).
We will see shortly that the converse holds (cf. Corollary~\ref{C:BrS2RS1}).

Every \rrep\ lattice is completely normal.
By Delzell and Madden~\cite{DelMad1994}, not every completely normal bounded distributive lattice can be represented in this way.
In fact, Mellor and Tressl established in~\cite{MelTre2012} the following result.

\begin{theorem}[Mellor and Tressl]\label{T:MelTre}
For every infinite cardinal~$\gl$, the class of all \rrep\ lattices cannot be defined by any class of~$\scL_{\infty,\gl}$-formulas of lattice theory.
\end{theorem}

The real spectrum can also be reduced to the Zariski spectrum, as follows.
A commutative unital ring~$A$ is \emph{real-closed} (cf. Schwartz \cite{Sch1986,Sch1989}, Prestel and Schwartz~\cite{PreSch2002}) if it has no nonzero nilpotent elements, the squares in~$A$ form the positive cone of a structure of \fring\ on~$A$, $0\leq a\leq b$ implies that $a^2\in Ab$, and for every prime ideal~$P$ of~$A$, the quotient field~$A(P)$ of~$A/P$ is real closed, and~$A/P$ is integrally closed in~$A(P)$.
Every commutative unital ring has a ``real closure''~$\Cr(A)$, which is a real-closed ring together with a unital ring homomorphism $A\to\Cr(A)$.
The following result is contained in Theorem~I.3.10,  Propositions~I.3.19 and~I.3.23, and the top of page~27, in Schwartz \cite{Sch1989}.

\begin{theorem}[Schwartz]\label{T:SchwRC}
For any commutative, unital ring~$A$, the canonical homomorphism $A\to\Cr(A)$ induces a homeomorphism $\Specr{\Cr(A)}\to\Specr{A}$.
Moreover, if~$A$ is real-closed, then the support map $P\mapsto P\cap(-P)$ induces a homeomorphism $\Specr{A}\to\Spec{A}$.
\end{theorem}

\begin{corollary}\label{C:SchwRC1}
For any commutative, unital ring~$A$, the Stone dual of the real spectrum of~$A$ is isomorphic to the lattice~$\Idrc{\Cr(A)}$ of all finitely generated radical ideals of~$\Cr(A)$.
\end{corollary}

Although the two following corollaries are probably well known, we could not find them explicitly stated anywhere, so we include proofs for convenience.

\begin{corollary}\label{C:SchwRC2}
Every closed subspace of a real spectrum is a real spectrum.
\end{corollary}

\begin{proof}
By Theorem~\ref{T:SchwRC}, every real spectrum has the form~$\Spec{A}$ for some real-closed ring~$A$.
By definition, any closed subspace of~$\Spec{A}$ has the form $\Spec(A,I)\eqdef\setm{P\in\Spec{A}}{I\subseteq P}$, for a subset~$I$ of~$A$, which we may assume to be a radical ideal of~$A$.
It follows that the assignment $P\mapsto P/I$ defines a homeomorphism from $\Spec(A,I)$ onto $\Spec(A/I)$.
Now it follows from Schwartz \cite[Theorem~I.4.5]{Sch1989} that the ring~$A/I$ is real-closed.
By the second part of Theorem~\ref{T:SchwRC}, it follows that $\Spec(A/I)$ is homeomorphic to $\Specr(A/I)$.
\end{proof}

\begin{corollary}\label{C:BrS2RS1}
The class of real spectra of all commutative unital rings and the class of Brumfiel spectra of all commutative unital \fring{s} are identical.
\end{corollary}

\begin{proof}
It follows from Theorem~\ref{T:RS2BS} that every real spectrum is the Brumfiel spectrum of some commutative unital \fring.
Conversely, for every commutative unital \fring~$A$, the assignment $P\mapsto A^++P$ defines a homeomorphism from~$\Specb{A}$ onto $\Specr(A,A^+)\eqdef\setm{Q\in\Specr{A}}{A^+\subseteq Q}$, with inverse the support map $Q\mapsto Q\cap(-Q)$ (cf. Madden and Schwartz \cite[page~49]{MadSch1999}).
Since~$\Specr(A,A^+)$ is, by definition, a closed subspace of~$\Specr{A}$, it follows form Corollary~\ref{C:SchwRC2} that~$\Specb{A}$ is the real spectrum of some commutative unital ring.
\end{proof}

While real-rep\-re\-sentabil\-ity makes sense only for bounded lattices, Brumfiel-rep\-re\-sentabil\-ity can also be defined for unbounded lattices.
According to the following corollary, the two concepts agree on bounded lattices.

\begin{corollary}\label{C:BrS2RS2}
A bounded distributive lattice is \rrep\ if{f} it is \brep.
\end{corollary}

\begin{proof}
By Stone duality and Corollary~\ref{C:BrS2RS1}, it suffices to prove that if a \emph{bounded} distributive lattice is \brep, then it can be represented by a commutative \emph{unital} \fring.
Let~$A$ be a commutative \fring\ such that~$\Idrc{A}$ has a top element.
Thus, $A=\seqr{u}_A$ for some~$u\in A^+$.
The localization~$A[u^{-1}]$ of~$A$ with respect to the multiplicative subset $S\eqdef\setm{u^n}{0<n<\go}$, endowed with the positive cone $A[u^{-1}]^+\eqdef\setm{x/u^n}{x\in A^+\,,\ 0<n<\go}$, is a commutative unital%
\footnote{This also applies to the degenerate case where~$u$ is nilpotent, in which case $A[u^{-1}]$ collapses to the zero element.}
\fring, for which the canonical homomorphism $\gk\colon A\to A[u^{-1}]$, $x\mapsto xu/u$ is an \fring\ homomorphism.
Obviously, the induced map $\bgk\colon\Idrc{A}\to\Idrc{A[u^{-1}]}$ is surjective.
We claim that~$\bgk$ is one-to-one.
Let $x,y\in A^+$ such that $\bgk\pI{\seqr{x}_A}\subseteq\bgk\pI{\seqr{y}_A}$.
This means that $\seqr{\gk(x)}_{A[u^{-1}]}\subseteq\seqr{\gk(y)}_{A[u^{-1}]}$\,, that is, there are $z\in A^+$ and positive integers~$m$, $n$ such that $\gk(x)^m\leq(z/u^n)\gk(y)$.
It follows that there are positive integers~$k$, $l$ such that $x^mu^k\leq zyu^l$ within~$A$.
Since $\seqr{u^k}_A=\seqr{u}_A=A$, we obtain, using~\eqref{Eq:jjmmseqr},
 \begin{equation*}
 \seqr{x}_A=\seqr{x^m}_A=\seqr{x^m}_A\cap\seqr{u^k}_A=
 \seqr{x^mu^k}_A\subseteq\seqr{zyu^l}_A\subseteq
 \seqr{y}_A\,,
 \end{equation*}
as required.
Therefore, $\Idrc{A}\cong\Idrc{A[u^{-1}]}$, with~$A[u^{-1}]$ a commutative unital \fring.
\end{proof}

\section{Counterexamples constructed from condensates}\label{S:DiagLift}

In the present section we shall apply the construction of a \emph{condensate}, put to use in the author's paper~\cite{CoordCX} for one arrow and studied in depth in Gillibert and Wehrung~\cite{Larder} for more complicated diagrams.
This construction enables us to construct non-representable \emph{objects} from non-representable \emph{arrows} (cf. Theorems~\ref{T:RCRnotlrep} and~\ref{T:NonHomrealSpec}), and it runs as follows.

\begin{definition}\label{D:Cond}
Let~$A$ and~$B$ be universal algebras (in a given similarity type) and let~$I$ be a set.
The \emph{$I$-condensate} of a homomorphism $\gf\colon A\to B$ is the following subalgebra of~$A\times B^I$:
 \[
 \Cond(\gf,I)\eqdef\setm{(x,y)\in A\times B^I}{y_i=\gf(x)
 \text{ for all but finitely many }i}\,.
 \]
\end{definition}

\begin{lemma}\label{L:CondasColimit}
The condensate $\Cond(\gf,I)$ is a directed union of copies of algebras of the form $A\times B^J$ for finite $J\subseteq I$.
\end{lemma}

\begin{proof}
For each finite $J\subseteq I$, denote by~$C_J$ the subalgebra of $\Cond(\gf,I)$ consisting of all pairs~$(x,y)\in A\times B^I$ such that~$y$ is constant on $I\setminus J$, with value~$\gf(x)$.
Then~$\Cond(\gf,I)$ is the directed union of all~$C_J$.
Clearly, $C_J\cong A\times B^J$.
\end{proof}

In particular, it follows from Lemma~\ref{L:CondasColimit} that if~$A$ and~$B$ are members of a class~$\cC$ of algebras, closed under nonempty finite direct products and directed colimits, then so is the condensate~$\Cond(\gf,I)$.
For example:

\begin{itemize}
\item
Whenever~$A$ and~$B$ are distributive lattices with zero and~$\gf$ is a $0$-lattice homomorphism, then~$\Cond(\gf,I)$ is a distributive lattice with zero.

\item
Whenever~$A$ and~$B$ are Abelian \lgrp{s} and~$\gf$ is an \lhom, then $\Cond(\gf,I)$ is an Abelian \lgrp.

\item
Whenever~$A$ and~$B$ are real-closed rings and~$\gf$ is a unital ring homomorphism, then $\Cond(\gf,I)$ is a real-closed ring.
 
 \end{itemize}
 
The proof of the following lemma is an extension of the one of Wehrung~\cite[The\-o\-rem~9.3]{CoordCX}.
 It is also an instance of a much more general result, called the \emph{Condensate Lifting Lemma} (\emph{CLL}), established in Gillibert and Wehrung~\cite{Larder}, that enables to infer representability of lat\-tice-indexed diagrams from representability of certain larger objects, also called there condensates.
In particular, as we will see shortly, Lemma~\ref{L:ellLift} can be extended to (many) other functors than~$\Idlc$\,.
The use of~CLL requires quite an amount of technical steps and although it may be unavoidable for complicated diagrams, the case of one arrow can be handled directly; thus we include here a self-contained proof for convenience.
 
\begin{lemma}\label{L:ellLift}
Let~$A$ and~$B$ be distributive lattices with zero, with~$A$ countable, and let $\gf\colon A\to B$ be a non-\lrep\ $0$-lattice homomorphism.
Then for any set~$I$, if the lattice homomorphism~$\gf$ is \lrep, then the distributive lattice~$\Cond(\gf,I)$ is \lrep.
If~$I$ is uncountable, then the converse holds.
\end{lemma}

\begin{proof}
Suppose first that~$\gf$ is \lrep.
This means that there are Abelian \lgrp{s}~$G$ and~$H$, together with an \lhom~$f\colon G\to H$, such that~$\Idlc{f}\cong\gf$ (as arrows).
Since the functor~$\Idlc$ commutes with nonempty finite direct products and with directed colimits, it sends the representation
 \[
 \Cond(f,I)=\varinjlim_{J\subset I\text{ finite}}(G\times H^J)\,,
 \]
given by Lemma~\ref{L:CondasColimit}, to a representation
 \[
 \Idlc\Cond(f,I)\cong\varinjlim_{J\subset I\text{ finite}}(A\times B^J)
 \cong\Cond(\gf,I)\,.
 \]
Hence, $\Cond(\gf,I)$ is \lrep.

Suppose, conversely, that~$I$ is uncountable and the distributive lattice~$E\eqdef\Cond(\gf,I)$ is \lrep.
This means that there are an Abelian \lgrp~$G$ and an isomorphism $\eps\colon\Idlc{G}\to E$.
Denote by $p\colon E\twoheadrightarrow A$ and $q_i\colon E\twoheadrightarrow B$, for $i\in I$, the canonical projections from~$E$.
The set $U\eqdef\setm{x\in G}{(p\circ\eps)(\seql{x}_{G})=0}$ is an $\ell$-ideal of~$G$; denote by $u\colon G\twoheadrightarrow G/U$ the canonical projection.
Since~$p$ is easily seen to be a closed map (cf. Definition~\ref{D:closedmap}), it follows from Lemma~\ref{L:FactClosed} that there is a unique isomorphism $\ga\colon\Idlc(G/U)\to A$ such that
 \begin{equation}\label{Eq:pe=au}
 p\circ\eps=\ga\circ(\Idlc{u})\,.
 \end{equation}
Likewise, for every $i\in I$, the set $V_i\eqdef\setm{x\in G}{(q_i\circ\eps)(\seql{x}_{G})=0}$ is an $\ell$-ideal of~$G$;
denote by $v_i\colon G\twoheadrightarrow G/V_i$ the canonical projection.
As for~$U$ and~$\ga$, there is a unique isomorphism $\gb_i\colon\Idlc(G/V_i)\to B$ such that
 \begin{equation}\label{Eq:qie=avi}
 q_i\circ\eps=\gb_i\circ(\Idlc{v_i})\,.
 \end{equation}
Since~$A$ is countable, a standard L\"owenheim-Skolem type argument shows that there exists a countable $\ell$-subgroup~$H$ of~$G$ such that~$\ga$ induces an isomorphism $\ga'\colon\Idlc(H/U)\to A$ (where we set $H/U\eqdef\setm{x+U}{x\in H}$).
We are going to argue that for every index~$i$ outside a certain countable subset of~$I$, the \lhom\ $f\colon H/U\to G/V_i$\,, $x+U\mapsto x+V_i$ is well-defined and represents the lattice homomorphism $\gf\colon A\to B$.
Our argument is partly illustrated in Figure~\ref{Fig:ellLift}.
The diagram of Figure~\ref{Fig:ellLift} is not commutative, because $\gf\circ p\neq q_i$ as a rule.

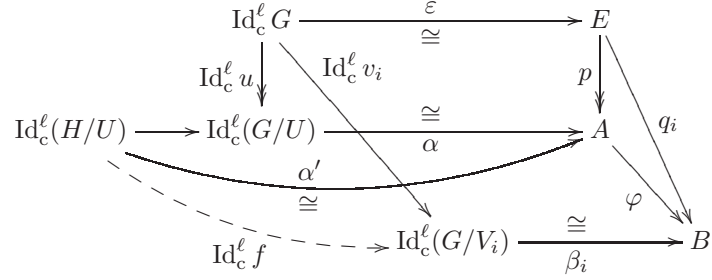
\begin{figure}[htb]
 \[
 \def\labelstyle{\displaystyle}
 \xymatrix{
 &\Idlc{G}\ar[rr]^{\eps}_{\cong}\ar@{->>}[d]_{\Idlc{u}}
 \ar[ddr]^(.3){\!\Idlc{v_i}}
 && E\ar@{->>}[d]_p\ar[ddr]^{q_i} &\\
 \Idlc(H/U)\ar[r]\ar@/_1.4pc/@{-->}[rrd]_{\Idlc{f}}
 \ar@/_1.8pc/[rrr]^(.45){\ga'}_(.45){\cong} &
 \Idlc(G/U)\ar[rr]^{\cong}_{\ga} && A\ar[dr]_{\gf} &\\
 && \Idlc(G/V_i)\ar[rr]^{\cong}_{\gb_i} && B 
 }
 \]
\caption{Illustrating the proof of Lemma~\ref{L:ellLift}}
\label{Fig:ellLift}
\end{figure}

For each $x\in H$, the element $\eps(\seql{x}_{G})$ belongs to~$E$, that is, $(\gf\circ p)\pI{\eps(\seql{x}_{G})}=q_i\pI{\eps(\seql{x}_{G})}$ for all but finitely many~$i$.
Since~$H$ is countable, there exists a countable subset~$J$ of~$I$ such that
 \begin{equation}\label{Eq:fp=qiae}
 (\gf\circ p)\pI{\eps(\seql{x}_{G})}=q_i\pI{\eps(\seql{x}_{G})}
 \text{ for all }x\in H\text{ and all }i\in I\setminus J\,.
 \end{equation}
Pick $i\in I\setminus J$.
We claim that $H\cap U\subseteq V_i$.
Every $x\in H\cap U$ satisfies the equations
 \[
 (\Idlc{u})\pI{\seql{x}_{G}}=\seql{x+U}_{G/U}=0\,,
 \]
thus, using~\eqref{Eq:pe=au}, $p\pI{\eps(\seql{x}_{G})}=(\ga\circ\Idlc{u})\pI{\seql{x}_{G}}=0$.
By~\eqref{Eq:fp=qiae}, it follows that $q_i\pI{\eps(\seql{x}_{G})}=0$.
By~\eqref{Eq:qie=avi}, this means that $\gb_i\pI{\seql{x+V_i}_{G/V_i}}=0$, that is, since~$\gb_i$ is an isomorphism, $x\in V_i$, thus completing the proof of our claim.

It follows that there exists a unique \lhom\ $f\colon H/U\to G/V_i$ such that $f(x+U)=x+V_i$ for every $x\in H$.
For every $x\in H$,
 \begin{align*}
 (\gf\circ\ga')\pI{\seql{x+U}_{H/U}}&=
 (\gf\circ\ga)\pI{\seql{x+U}_{G/U}}\\
 &=(\gf\circ\ga\circ\Idlc{u})\pI{\seql{x}_G}\\
 &=(\gf\circ p\circ\eps)\pI{\seql{x}_G}
 &&(\text{use~\eqref{Eq:pe=au}})\\
 &=(q_i\circ\eps)\pI{\seql{x}_G}
 &&(\text{use~\eqref{Eq:fp=qiae}})\\
 &=(\gb_i\circ\Idlc{v_i})\pI{\seql{x}_G}
 &&(\text{use~\eqref{Eq:qie=avi}})\\
 &=\gb_i(\seql{x+V_i}_{G/V_i})\\
 &=(\gb_i\circ\Idlc{f})\pI{\seql{x+U}_{H/U}}\,,
 \end{align*}
so~$\gf\circ\ga'=\gb_i\circ\Idlc{f}$.
Therefore, $f$ represents~$\gf$.
\end{proof}

\begin{theorem}\label{T:RCRnotlrep}
There exists a real-closed ring, of cardinality~$\aleph_1$, whose real spectrum is not homeomorphic to the $\ell$-spectrum of any Abelian \lgrp.
\end{theorem}

\begin{proof}
Let~$K$ be any countable, non-Archimedean real-closed field.
The subset
 \[
 A\eqdef\setm{x\in K}
 {(\exists n<\go)(-n\cdot 1_K\leq x\leq n\cdot 1_K)}
 \]
is an order-convex unital subring of~$K$, thus it is a real-closed ring.
Hence, denoting by $\eps\colon A\hookrightarrow K$ the inclusion map, it follows from Lemma~\ref{L:CondasColimit} that the condensate $R\eqdef\Cond(\eps,\go_1)$ is a real-closed ring.
Observe that the cardinality of~$R$ is~$\aleph_1$.

Suppose that the real spectrum of~$R$ is homeomorphic to the $\ell$-spectrum of an Abelian \lgrp~$G$.
By Proposition~\ref{T:Stone}, it follows that $\Idrc{R}\cong\Idlc{G}$.
Since the functor~$\Idrc$ commutes with nonempty finite direct products and directed colimits, it follows from Lemma~\ref{L:CondasColimit} that $\Idrc{R}\cong\Cond(\Idrc{\eps},\go_1)$.
In particular, the distributive lattice~$\Cond(\Idrc{\eps},\go_1)$ is \lrep.
By Lemma~\ref{L:ellLift}, it follows that the lattice homomorphism~$\Idrc\eps$ is \lrep.
By Lemma~\ref{L:Idcf}, it follows that the map $\beps\eqdef\Idrc\eps$ is closed.

However, $\Idrc{A}$ is a chain with more than two elements, $\Idrc{K}$ is the two-element chain, and~$\beps$ is the unique zero-separating map $\Idrc{A}\to\Idrc{K}$.
In particular, if $0<u<1$ in~$\Idrc{A}$, then $\beps(1)=\beps(u)\vee0$ but there is no $x\in\Idrc{A}$ such that $1\leq u\vee x$ and $\beps(x)\leq0$.
Hence, $\beps$ is not closed.
\end{proof}

\begin{remark}\label{Rk:RCRnotlrep}
By (the proof of) Dickmann, Gluschankof, and Lucas \cite[Proposition~1.1]{DicGlu2000}, the field~$K$ and the ring~$A$, of the proof of Theorem~\ref{T:RCRnotlrep}, can be constructed  in such a way that~$\Idrc{A}$ is the three-element chain.
By a simple L\"owenheim-Skolem type argument, $K$ may be taken countable.
Then the lattice $\Cond(\Idrc{\eps},\go_1)$ is isomorphic to the lattice~$\bD_{\go_1}$ introduced in Wehrung \cite[\S~5]{MV1}.
As observed in the final example of Wehrung \cite[\S~5]{MV1}, $\bD_{\go_1}$ is a homomorphic image of an \lrep\ distributive lattice, without being \lrep\ itself.
By Theorem~\ref{T:StoneDual}, this means that the spectrum of~$\bD_{\go_1}$ can be embedded, as a spectral subspace, into the $\ell$-spectrum of an Abelian \lgrp, without being an $\ell$-spectrum itself.
\end{remark}

The negative property satisfied by the counterexample~$R$ of Theorem~\ref{T:RCRnotlrep} cannot be strengthened further by replacing ``$\ell$-spectrum'' by ``spectral subspace of an $\ell$-spectrum''.
The reason for this is the following easy observation, which ought to be well known but for which we could not locate a reference.

\begin{proposition}\label{P:Idl2Idr}
For every commutative \fring~$A$, there are an Abelian \lgrp~$G$, which can be taken with order-unit if~$A$ is unital, together with a surjective lattice homomorphism $\mu\colon\Idlc{G}\twoheadrightarrow\Idrc{A}$.
Hence, every \brep\ lattice is a homomorphic image of some \lrep\ lattice.
\end{proposition}

\begin{proof}
Denote by~$G$ the underlying additive \lgrp\ of~$A$.
It is easy to verify that the map $\Idlc{G}\twoheadrightarrow\Idrc{A}$, $\seql{x}\mapsto\seqr{x}$ is a well defined lattice homomorphism (use~\eqref{Eq:jjmmseql} and~\eqref{Eq:jjmmseqr}).
It is, of course, surjective.

Now suppose that~$A$ is unital.
Since the multiplicative unit~$1$ of~$A$ may not be an order-unit of~$A$, the construction above of~$G$ must be modified.
To this end, define~$G$ as the underlying \lgrp\ of $\setm{x\in A}{(\exists n<\go)(-n\cdot 1\leq x\leq n\cdot 1)}$, and define again $\mu\colon\seql{x}\mapsto\seqr{x}$.
We need to prove that~$\mu$ is surjective.
For every $x\in A^+$, the relations $x=(x\vee1)\cdot(x\wedge1)$ and $0\leq x\wedge1\leq x$ imply that~$x$ and~$x\wedge1$ generate the same $\ell$-ideal, and thus, \emph{a fortiori}, $\seqr{x}=\seqr{x\wedge1}=\mu(\seql{x\wedge1})$.
\end{proof}

By applying Theorem~\ref{T:StoneDual} and Remark~\ref{Rk:StoneDual} to Proposition~\ref{P:Idl2Idr}, we thus get

\begin{corollary}\label{C:Idl2Idr}
The Brumfiel spectrum of any commutative unital \fring\ \pup{thus also every real spectrum} is a spectral subspace of the $\ell$-spectrum of some Abelian \lgrp\ with order-unit.
\end{corollary}

Moving to \fring{s}, a \emph{mutatis mutandis} modification of the proof of Lemma~\ref{L:ellLift}, using Lemma~\ref{L:FactClosedRng} instead of Lemma~\ref{L:FactClosed}, leads to the following result.

\begin{lemma}\label{L:effLift}
Let~$A$ and~$B$ be distributive lattices with zero, with~$A$ countable, and let $\gf\colon A\to B$ be a non-\brep\ $0$-lattice homomorphism.
Then for any set~$I$, if the lattice homomorphism~$\gf$ is \brep, then the distributive lattice~$\Cond(\gf,I)$ is \brep.
If~$I$ is uncountable, then the converse holds.
\end{lemma}

\begin{theorem}\label{T:NonHomrealSpec}
There exists a \rrep\ distributive lattice~$E$, of cardinality~$\aleph_1$, with a non-\brep\ \pup{thus non-\rrep} homomorphic image.
\end{theorem}

\begin{proof}
It follows from Dickmann, Gluschankof, and Lucas \cite[Proposition~1.1]{DicGlu2000} that there exists a real-closed domain~$A$ with exactly three prime ideals $P_1\subset P_2\subset P_3$.
Hence, $\Idrc{A}$ is isomorphic to the four-element chain $\four\eqdef\set{0,1,2,3}$.
By a simple L\"owenheim-Skolem type argument, $A$ may be taken countable.
Denote by~$R$ the ring of all almost constant families $\vecm{x_{\xi}}{\xi<\go_1}$ of elements of~$A$.
Then the lattice~$E$ of all almost constant $\go_1$-sequences of elements in~$\four$ is isomorphic to~$\Idrc{R}$, thus it is \rrep.

Consider the chains $\three\eqdef\set{0,1,2}$ and $\four\eqdef\set{0,1,2,3}$.
The (surjective) dual homomorphism $\gf\colon\four\twoheadrightarrow\three$, of the map $\set{1,2}\to\set{1,2,3}$ sending~$1$ to~$1$ and~$2$ to~$3$, is not convex, thus, by Lemma~\ref{L:BrumConvex}, not \brep.
By Lemma~\ref{L:effLift}, it follows that the lattice $\Cond(\gf,\go_1)
$ is not \brep.

On the other hand, the assignment $\vecm{x_{\xi}}{\xi<\go_1}\mapsto(x_{\infty},\vecm{\gf(x_{\xi})}{\xi<\go_1})$ defines a surjective lattice homomorphism from~$E$ onto $\Cond(\gf,\go_1)$.
\end{proof}

By Stone duality (cf. Theorem~\ref{T:StoneDual} and Remark~\ref{Rk:StoneDual}), it follows that the spectrum of the bounded distributive lattice $\Cond(\gf,\go_1)$ witnesses the following corollary.

\begin{corollary}\label{C:NonHomrealSpec}
There exists a real spectrum with a spectral subspace which is not a Brumfiel spectrum \pup{thus also not a real spectrum}.
\end{corollary}

\section{An \lrep, non-Brumfiel representable lattice}\label{S:NonrRepr}

Although the proof of the present section's main negative result, Theorem~\ref{T:NonrReprlGrp}, arises from a lattice-theoretical investigation of the argument of Delzell and Madden \cite[Lemma~2]{DelMad1994}, the construction of its counterexample, which is the \lgrp\ that we will denote by~$G_{\go_1^{\op}}$\,, is somehow simpler.
Moreover, the constructions of Delzell and Madden~\cite{DelMad1994}, Mellor and Tressl~\cite{MelTre2012} yield lattices with~$2^{\aleph_1}$ elements \emph{a priori}, while our construction yields the smaller size~$\aleph_1$.

\begin{notation}\label{Not:FGgL}
We denote by~$F$ the free Abelian \lgrp\ defined by generators~$a$ and~$b$, subjected to the relations $a\geq0$ and~$b\geq0$.
Moreover, we set $G_{\gL}\eqdef\ZZ\seq{\gL}\lextimes F$, for any chain~$\gL$.
\end{notation}

Since~$\ZZ\seq{\gL}$ is a totally ordered group and~$F$ is an Abelian \lgrp, $G_{\gL}$ is also an Abelian \lgrp.
It has an order-unit if{f}~$\gL$ has a largest element.
We shall occasionally identify~$F$ with the $\ell$-ideal $\set{0}\times F$ of~$G_{\gL}$.

\begin{lemma}\label{L:a+bindec}
The $\ell$-ideal $\seql{a+b}$ is directly indecomposable in the lattice~$\Idlc{G_{\gL}}$.
That is, there are no nonzero $\bx,\by\in\Idlc{G_{\gL}}$ such that $\seql{a+b}=\bx\vee\by$ and $\bx\cap\by=\set{0}$.
\end{lemma}

\begin{proof}
Since $a,b\in F$ and~$F$ is an ideal of~$G_{\gL}$, it suffices to prove that~$\seql{a+b}$ is directly indecomposable in the lattice~$\Idlc{F}$.
The right closed upper quadrant
 \[
 \gO\eqdef\setm{(x,y)\in\RR^2}{x\geq0\text{ and }y\geq0}
 \]
is a convex subset of~$\RR^2$.
Further, by the Baker-Beynon duality (cf. Baker~\cite{Baker1968}, Beynon \cite{Beyn1975,Beyn1977}), $F$ is isomorphic to the $\ell$-subgroup of~$\RR^{\gO}$ generated by the canonical projections $a\colon(x,y)\mapsto x$ and $b\colon(x,y)\mapsto y$, and there exists a unique lattice embedding~$\iota$ from~$\Idlc{F}$
to the lattice of all relative open subsets of~$\gO$ such that $\iota(\seql{x})=\bck{x\neq0}$ whenever $x\in F$.
Hence, in order to prove that~$\seql{a+b}$ is indecomposable in~$\Idlc{F}$, it suffices to prove that $\bck{a+b\neq0}$ is a connected subset of~$\RR^2$.
This, in turn, follows from the relation $\bck{a+b\neq0}=\setm{(x,y)\in\gO}{x+y>0}$, which implies that $\bck{a+b\neq0}$ is a convex (thus connected) subset of~$\gO$, thus of~$\RR^2$.
\end{proof}

\begin{lemma}\label{L:Asymp}
Let~$\gL$ be a chain.
Then every pairwise orthogonal subset of~$G_{\gL}^{++}$ is countable.
\end{lemma}

\begin{proof}
Since~$\ZZ\seq{\gL}$ is a chain, every pairwise orthogonal subset~$X$ of~$G_{\gL}^{++}$, with more than one element, is a subset of~$F$.
The latter being countable, $X$ is countable.
\end{proof}

\begin{theorem}\label{T:NonrReprlGrp}
There are no commutative \fring~$A$ and no surjective lattice homomorphism $\mu\colon\Idrc{A}\twoheadrightarrow\Idlc{G_{\go_1^{\op}}}$.
\end{theorem}

Until the end of the proof of Theorem~\ref{T:NonrReprlGrp}, we shall assume, by way of contradiction, that there are a commutative \fring~$A$ and a surjective lattice homomorphism $\mu\colon\Idrc{A}\twoheadrightarrow\Idlc{G_{\go_1^{\op}}}$.

Pick $x,y\in A^+$ such that $\seql{a}=\mu\seqr{x}$ and $\seql{b}=\mu\seqr{y}$.
Moreover, for each $\xi<\go_1$, pick $z_{\xi}\in A^+$ such that
 \begin{equation}\label{Eq:cga2zga}
 \seql{c_{\xi}}=\mu\seqr{z_{\xi}}\,.
 \end{equation}
In particular,
 \begin{equation}\label{Eq:cleqzga}
 \seql{a+b}\subseteq\mu\seqr{z_{\xi}}\,,\quad
 \text{whenever }\xi<\go_1\,.
 \end{equation}
Let $\xi<\eta<\go_1$.
Since $z_{\xi}\leq z_{\eta}+(z_{\xi}\sd z_{\eta})$, with~$z_{\xi}$, $z_{\eta}$, and~$z_{\xi}\sd z_{\eta}$ all in~$A^+$, we get
 \[
 \seqr{z_{\xi}}\subseteq
 \seqr{z_{\eta}}\vee\seqr{z_{\xi}\sd z_{\eta}}\,,
 \]
whence, applying the lattice homomorphism~$\mu$ and by~\eqref{Eq:cga2zga},
 \[
 c_{\xi}\in\seql{c_{\eta}}\vee
 \mu\seqr{z_{\xi}\sd z_{\eta}}\,.
 \]
Since $c_{\eta}\ll c_{\xi}$ (within~$G_{\go_1^{\op}}$), it follows that
 \[
 c_{\xi}\in\mu\seqr{z_{\xi}\sd z_{\eta}}\,.
 \]
Hence we obtain, \emph{a fortiori}, that
 \begin{equation}\label{Eq:cleqzga-gb}
 \seql{a+b}\subseteq\mu\seqr{z_{\xi}\sd z_{\eta}}\,,\quad
 \text{whenever }\xi<\eta<\go_1\,.
 \end{equation}

\begin{lemma}\label{L:zxidecr}
For all $\xi\leq\eta<\go_1$, the relation
$\seql{a+b}\cap\mu\seqr{z_{\eta}\sd z_{\xi}}=\set{0}$ holds.
\end{lemma}

\begin{proof}
The conclusion is trivial if $\xi=\eta$.
Now suppose that $\xi<\eta$.
We compute:
 \begin{align}
 \seql{a+b}\cap\mu\seqr{z_{\eta}\sd z_{\xi}}&\subseteq
 \mu\seqr{z_{\xi}\sd z_{\eta}}\cap\mu\seqr{z_{\eta}\sd z_{\xi}}
 &&\text{(use~\eqref{Eq:cleqzga-gb})}\notag\\
 &=\mu\pI{\seqr{z_{\xi}\sd z_{\eta}}\cap\seqr{z_{\eta}\sd z_{\xi}}}
 \notag\\
 &=\mu(\set{0})
 &&(\text{use~\eqref{Eq:jjmmseqr} and Lemma~\ref{L:Gabcsd}})\notag\\
 &=\set{0}\,.\tag*{\qed}
 \end{align}
\renewcommand{\qed}{}
\end{proof}

For each $\xi<\go_1$, we set
 \begin{align*}
 x_{\xi}&\eqdef(x+y)\wedge\pI{xz_0\sd(x+y)z_{\xi+1}}\,,\\
 y_{\xi}&\eqdef(x+y)\wedge\pI{(x+y)z_{\xi}\sd xz_0}\,,\\
 \bx_{\xi}&\eqdef\mu\seqr{x_{\xi}}\,,\\
 \by_{\xi}&\eqdef\mu\seqr{y_{\xi}}\,.
 \end{align*}

\begin{lemma}\label{Eq:Ineqxyga}
The following relations hold, whenever $\xi<\eta<\go_1$:
\begin{enumerater}
\item\label{xgasmall}
$\bx_{\xi}\subseteq\seql{a}$;

\item\label{ygasmall}
$\by_{\xi}\subseteq\seql{b}$;

\item\label{a+bleqxvyga}
$\seql{a+b}=\bx_{\xi}\vee\by_{\xi}$;

\item\label{xygagbdisj}
$\bx_{\xi}\cap\by_{\eta}=\set{0}$.

\end{enumerater}
\end{lemma}

\begin{proof}
\emph{Ad}~\eqref{xgasmall}.
{}From $x+y\geq0$ and $z_{\xi}\geq0$ it follows that $0\leq x_{\xi}\leq xz_0$, whence $\seqr{x_{\xi}}\subseteq\seqr{xz_0}\subseteq\seqr{x}$.
Apply the homomorphism~$\mu$.

\emph{Ad}~\eqref{ygasmall}.
{}From~\eqref{Eq:cleqzga-gb} it follows that
$\seq{a+b}\subseteq\mu\seqr{z_{0}\sd z_{\xi}}$.
Hence,
 \begin{align}
 \seq{a+b}\cap\mu\seqr{xz_{\xi}\sd xz_{0}}&\subseteq
 \seq{a+b}\cap\mu\seqr{z_{\xi}\sd z_{0}}\label{Eq:a+b1}\\
 &=\set{0}&&\text{by Lemma~\ref{L:zxidecr}}\,.\label{Eq:a+b4}
 \end{align}
Using Lemma~\ref{L:Gabcsd}, we get $(x+y)z_{\xi}\sd xz_0\leq yz_{\xi}+(xz_{\xi}\sd xz_0)$, thus
 \[
 y_{\xi}\leq(x+y)\wedge\pI{yz_{\xi}+(xz_{\xi}\sd xz_0)}\,,
 \]
and thus, applying the homomorphism~$\mu$ together with~\eqref{Eq:a+b1}--\eqref{Eq:a+b4},
 \[
 \by_{\xi}=\mu\seqr{y_{\xi}}\subseteq
 \mu\seqr{(x+y)\wedge yz_{\xi}}\vee\pII{
 \seql{a+b}\cap\mu\seqr{xz_{\xi}\sd xz_0}}
 \subseteq\mu\seqr{y}=\seql{b}\,.
 \]

\emph{Ad}~\eqref{a+bleqxvyga}.
We compute:
 \begin{align*}
 \seql{a+b}&=\mu\pI{\seqr{x+y}\cap\seqr{z_{\xi}\sd z_{\xi+1}}}
 &&(\text{apply~\eqref{Eq:cleqzga-gb}})\\
 &=\mu\seqr{(x+y)(z_{\xi}\sd z_{\xi+1})}
 &&(\text{use~\eqref{Eq:jjmmseqr}})\\
 &=\mu\seqr{(x+y)z_{\xi}\sd (x+y)z_{\xi+1}}
 &&(\text{use Lemma~\ref{L:Gabcsd}})\\
 &\subseteq\mu\seqr{(x+y)z_{\xi}\sd xz_{0}}\vee
 \mu\seqr{xz_{0}\sd (x+y)z_{\xi+1}}
 &&(\text{use Lemma~\ref{L:Gabcsd}})\,.
 \end{align*}
Since $\seql{a+b}=\mu\seqr{x+y}$, the desired conclusion follows from the distributivity of the lattice~$\Idlc{G_{\go_1^{\op}}}$\,.
 
\emph{Ad}~\eqref{xygagbdisj}.
It follows from Lemma~\ref{L:zxidecr} that
$\seq{a+b}\cap\mu\seqr{z_{\eta}\sd z_{\xi+1}}=\set{0}$ and hence, \emph{a fortiori}, that
 \begin{equation}\label{Eq:a+bcapzet0=0}
 \seq{a+b}\cap\mu\seqr{(x+y)z_{\eta}\sd(x+y)z_{\xi+1}}=\set{0}\,.
 \end{equation}
Hence, using Lemma~\ref{L:Gabcsd},
 \begin{align*}
 \by_{\eta}&=
 \seq{a+b}\cap\mu\seqr{(x+y)z_{\eta}\sd xz_{0}}\\
 &\subseteq
 \seq{a+b}\cap\pII{\mu\seqr{(x+y)z_{\eta}\sd(x+y)z_{\xi+1}}
 \vee\mu\seqr{(x+y)z_{\xi+1}\sd xz_{0}}}\\
 &=\seq{a+b}\cap\mu\seqr{(x+y)z_{\xi+1}\sd xz_{0}}
 \qquad\text{(use~\eqref{Eq:a+bcapzet0=0})}\,.
 \end{align*}
It follows that
 \begin{align}
 \bx_{\xi}\cap\by_{\eta}&\subseteq
 \mu\seqr{xz_{0}\sd(x+y)z_{\xi+1}}\cap
 \mu\seqr{(x+y)z_{\xi+1}\sd xz_{0}}\notag\\
 &=\mu(\set{0})\qquad
 (\text{use~\eqref{Eq:jjmmseqr} and Lemma~\ref{L:Gabcsd}})
 \notag\\
 &=\set{0}\,.\tag*{\qed} 
 \end{align}
\renewcommand{\qed}{}
\end{proof}

Set $u_{\xi}\eqdef x_{\xi}\wedge y_{\xi}$ and $\bu_{\xi}\eqdef\mu\seqr{u_{\xi}}$, for all $\xi<\go_1$.

\begin{lemma}\label{L:uxineq0}
The principal $\ell$-ideals~$\bu_{\xi}$ of~$G_{\go_1^{\op}}$, for $\xi<\go_1$, are all nonzero and pairwise orthogonal.
\end{lemma}

\begin{proof}
The statement of pairwise orthogonality follows from Lemma~\ref{Eq:Ineqxyga}\eqref{xygagbdisj}.
If $\bu_{\xi}=\nobreak\set{0}$, then $\bx_{\xi}\cap\by_{\xi}=\set{0}$, thus, by Lemmas~\ref{L:a+bindec} and~\ref{Eq:Ineqxyga}\eqref{a+bleqxvyga}, either~$\bx_{\xi}=\set{0}$ or~$\by_{\xi}=\set{0}$, thus either~$\bx_{\xi}=\seql{a+b}$ or~$\by_{\xi}=\seql{a+b}$, and thus, by items~\eqref{xgasmall} and~\eqref{ygasmall} of Lemma~\ref{Eq:Ineqxyga}, either $a+b\in\seql{a}$ or $a+b\in\seql{b}$, a contradiction.
\end{proof}

\begin{proof}[End of the proof of Theorem~\ref{T:NonrReprlGrp}]
By Lemma~\ref{L:uxineq0}, $\Idlc{G_{\go_1^{\op}}}$ has an uncountable, pairwise orthogonal set of nonzero ideals.
Picking positive generators of the~$\bu_{\xi}$, we get an uncountable, pairwise orthogonal set of nonzero elements in~$G_{\go_1^{\op}}$, in contradiction with Lemma~\ref{L:Asymp}.
\end{proof}

As mentioned before, the real spectrum of any commutative unital ring~$A$ is homeomorphic to the Brumfiel spectrum of the universal \fring~of~$A$ (cf. Theorem~\ref{T:RS2BS}).
By applying Theorem~\ref{T:StoneDual}, we get the following corollary.

\begin{corollary}\label{C:NonrReprlGrp}
The $\ell$-spectrum of the unital Abelian \lgrp~$G_{\go_1^{\op}}$ cannot be embedded, as a spectral subspace, into the Brumfiel spectrum of any commutative \fring.
Hence, it also cannot be embedded, as a spectral subspace, into the real spectrum of any commutative unital ring.
\end{corollary}

\section{Omitting homomorphic images of \lrep\ lattices}\label{S:UnliftCube}

Although it is well known, since Delzell and Madden~\cite{DelMad1994}, that there are non-\lrep\ completely normal bounded distributive lattices, the corresponding result for \emph{homomorphic images} of \lrep\ lattices was not known.
In this section, we shall fill that gap by constructing a completely normal bounded distributive lattice, of cardinality~$\aleph_2$, which is not a homomorphic image of any \lrep\ lattice.
The method used, in particular the part involving Kuratowski's Free Set Theorem, originates in the author's paper~\cite{NonMeas}.
A crucial step consists of coining a property satisfied by all homomorphic images of \lrep\ lattices.

\begin{lemma}\label{L:RefI}
For every set~$I$, every homomorphic image of an \lrep\ lattice satisfies the following infinitary statement:
\begin{equation}\label{Eq:RefI}
 \begin{aligned}
 \text{For every family }\vecm{\ba_i}{i\in I}\,,
 \text{ there exists a family }
 \vecm{\bc_{i,j}}{(i,j)\in I\times I}\text{ such that}\\
 \text{ each }
 \ba_i=(\ba_i\wedge\ba_j)\vee\bc_{i,j}\,,
 \text{ each }\bc_{i,j}\wedge\bc_{j,i}=0\,,\text{ and each }
 \bc_{i,k}\leq\bc_{i,j}\vee\bc_{j,k}\,.
 \end{aligned}
 \end{equation}
\end{lemma}

\begin{proof}
Since~\eqref{Eq:RefI} is obviously preserved under homomorphic images, it suffices to prove that~$\Idlc{G}$ satisfies~\eqref{Eq:RefI}, whenever~$G$ is an Abelian \lgrp.
Write $\ba_i=\seql{a_i}$, where $a_i\in G^+$, for all $i\in I$.
The principal $\ell$-ideals $\bc_{i,j}\eqdef\seql{a_i\sd a_j}$, for $i,j\in I$, satisfy the required conditions.
\end{proof}

As usual, we set $\two\eqdef\set{0,1}$ and $\three\eqdef\set{0,1,2}$.
We set $\ol{0}=0$, $\ol{1}=\ol{2}=2$, we denote by $\be\colon\two\hookrightarrow\three$ the map sending~$0$ to~$0$ and~$1$ to~$2$, and we denote by $\br\colon\three\twoheadrightarrow\two$ the map sending~$0$ to~$0$ and any nonzero element to~$1$.
Let $\bbf,\bg\colon\three\hookrightarrow\three^2$ and
$\ba,\bb,\bc\colon\three^2\hookrightarrow\three^3\times\two$ be the maps defined by
 \begin{align*}
 \bbf(x)&=(\ol{x},x)\,,\\
 \bg(x)&=(x,\ol{x})\,, 
 \end{align*}
for all~$x\in\three$, and
 \begin{align*}
 \ba(x,y)&=(\ol{x},x,y,\br(y))\,,\\
 \bb(x,y)&=(x,\ol{x},y,\br(x))\,,\\
 \bc(x,y)&=(x,y,\ol{y},\br(y))\,, 
 \end{align*}
for all $(x,y)\in\three^2$.
The mappings~$\be$, $\bbf$, $\bg$, $\ba$, $\bb$, $\bc$ form a commutative diagram of finite distributive lattices with $0,1$-lattice embeddings, represented in Figure~\ref{Fig:UnliftCube} as indexed by the cube $\pI{\Pow(\set{1,2,3}),\subseteq}$.
Observe that none of the maps~$\bbf$, $\bg$, $\ba$, $\bb$, $\bc$ is closed.
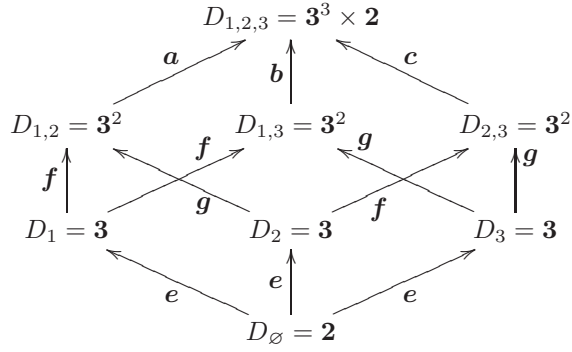
\begin{figure}[htb]
 \[
 \def\labelstyle{\displaystyle}
 \xymatrix{
 & D_{1,2,3}=\three^3\times\two & \\
 D_{1,2}=\three^2\ar[ur]^{\ba} & D_{1,3}=\three^2\ar[u]^{\bb}
 & D_{2,3}=\three^2\ar[ul]_{\bc}\\
 D_1=\three\ar[u]^{\bbf}\ar[ur]^(.6){\bbf\!\!\!} &
 D_2=\three\ar[ul]^(.4){\bg\!\!\!}\ar[ur]_(.4){\!\!\!\bbf} &
 D_3=\three\ar[u]_(.7){\bg}\ar[ul]_(.7){\bg}\\
 & D_\es=\two\ar[ul]^{\be}\ar[u]^{\be}\ar[ur]_{\be} &
  }
 \]
\caption{A cube of $0,1$-lattice embeddings}
\label{Fig:UnliftCube}
\end{figure}

 \begin{figure}[htb]
 \[
 \def\labelstyle{\displaystyle}
 \xymatrix{
 & E_{1,2} &\\
 E_1\ar[ru]^{\bg_1} && E_2\ar[lu]_{\bg_2}\\
 & E_0\ar[lu]^{\bbf_1}\ar[ru]_{\bbf_2} &
 }
 \]
 \caption{A commutative square for Lemma~\ref{L:StrAmalg}}
 \label{Fig:Dijpush}
 \end{figure}
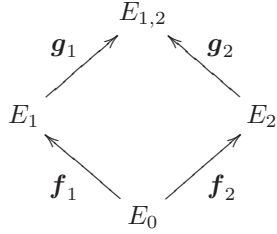
 
We leave to the reader the straightforward verification of the following lemma.

\begin{lemma}\label{L:StrAmalg}
Every square as on Figure~\textup{\ref{Fig:Dijpush}},
between three consecutive levels in the cube of Figure~\textup{\ref{Fig:UnliftCube}}, is a \emph{strong amalgam}, that is, setting $\bh\eqdef\bg_1\circ\bbf_1=\bg_2\circ\bbf_2$, the relation $\bg_1[E_1]\cap\bg_2[E_2]=\bh[E_0]$ holds.
\end{lemma}

\begin{notation}\label{Not:cVvar}
We denote by~$\cL$ the similarity type $(\vee,\wedge,\sd,0,1)$, where~$\vee$, $\wedge$, and~$\sd$ are binary operation symbols and~$0$, $1$ are constant symbols.
Moreover, we denote by~$\cV_0$ the variety of $\cL$-structures obtained by stating the identities defining bounded distributive lattices on $(\vee,\wedge,0,1)$, together with the identities
 \begin{align}
 (\vx\wedge\vy)\vee(\vx\sd\vy)&=\vx\,,\label{Eq:Diff1}\\
 (\vx\sd\vy)\wedge(\vy\sd\vx)&=0\,.\label{Eq:Diff2} 
 \end{align}
\end{notation}

Evidently, (the lattice reduct of) any member of~$\cV_0$ is a completely normal bounded distributive lattice.

\begin{lemma}\label{L:Cube2cVstruct}
The commutative diagram, of bounded lattice embeddings, represented in Figure~\textup{\ref{Fig:UnliftCube}}, can be expanded to a commutative diagram of embeddings in~$\cV_0$.
\end{lemma}

\begin{proof}
Represent the diagram by lattices~$D_p$ and arrows $\bbf_p^q\colon D_p\to D_q$, where $p\subseteq q\subseteq\set{1,2,3}$ (cf. Figure~\ref{Fig:UnliftCube}).
We define inductively the ``difference operation''~$\sd_{D_p}$ on~$D_p$, assuming that it has already been defined on all~$D_q$, for $q\subsetneqq p$.
Let $x_1,x_2\in D_p$.
If~$x_1$ and~$x_2$ belong to the range of~$\bbf_q^p$ for some $q\subsetneqq p$, then due to Lemma~\ref{L:StrAmalg}, there is a smallest such~$q$;
then let $y_1,y_2\in D_q$ such that each $x_i=\bbf_q^p(y_i)$, and define $x_1\sd_{D_p}x_2\eqdef\bbf_q^p(y_1\sd_{D_q}y_2)$.
If~$x_1$ and~$x_2$ do not belong to the range of any~$\bbf_q^p$ where $q\subsetneqq p$, then pick any splitting $(u,v)$ of $(x_1,x_2)$ in~$D_p$\,, then define $x_1\sd_{D_p}x_2\eqdef u$ and $x_2\sd_{D_p}x_1\eqdef v$.
By construction, $D_p$ is thus expanded to a member of~$\cV_0$, and each~$\bbf_q^p$ is an $\cL$-embedding.
\end{proof}

For the remainder of Section~\ref{S:UnliftCube}, we shall denote by $\vec{D}=\vecm{D_p,\bbf_p^q}{p\subseteq q\subseteq\set{1,2,3}}$ the cube of~$\cV_0$ obtained from Lemma~\ref{L:Cube2cVstruct}, and by $D\eqdef D_{1,2,3}=(\three^3\times\two,\vee,\wedge,\sd,0,1)$ its top member.
We also pick any variety~$\cV$ of $\cL$-algebras such that $D\in\cV$ and $\cV\subseteq\cV_0$, and we denote by~$\FV(X)$ the free $\cV$-algebra on~$X$, for any set~$X$.
We will identify~$\FV(X)$ with its canonical copy in~$\FV(Y)$, whenever $X\subseteq Y$.

\begin{theorem}\label{T:FVInotlrep}
For every set~$I$, the underlying lattice of~$\FV(I)$ is a completely normal bounded distributive lattice.
Moreover, if $\card{I}\geq\aleph_2$, then~$\FV(I)$ is not a homomorphic image of any \lrep\ lattice.
\end{theorem}

\begin{proof}
The first statement, that~$\FV(I)$ is completely normal, is obvious.
Suppose, from now on, that $\card{I}\geq\aleph_2$, and denote by~$a_i$, for $i\in I$, the canonical generators of~$\FV(I)$.
By Lemma~\ref{L:RefI}, it suffices to prove that~$\FV(I)$ does not satisfy the infinitary statement~\eqref{Eq:RefI}.
Suppose otherwise.
Then there exists a family $\vecm{c_{i,j}}{(i,j)\in I\times I}$, of elements in~$\FV(I)$, satisfying the conditions stated in~\eqref{Eq:RefI}.

For all $i,j\in I$, there exists a finite subset $\Phi(\set{i,j})$ of~$I$ such that $\set{c_{i,j},c_{j,i}}\subseteq\FV\pI{\Phi(\set{i,j})}$.
Since $\card{I}\geq\aleph_2$ and by Kuratowski's Free Set Theorem (cf. Kuratowski~\cite{Kurat1951}, Erd\H{o}s \emph{et al.} \cite[Theorem~46.1]{EHMR}), there are distinct elements in~$I$, that we may denote by~$1$, $2$, $3$, such that
 \begin{equation}\label{Eq:123indep}
 1\notin\Phi(\set{2,3})\,,\ 2\notin\Phi(\set{1,3})\,,\text{ and }
 3\notin\Phi(\set{1,2})\,.
 \end{equation}
Then $J\eqdef\set{1,2,3}$ is a subset of~$I$.
We shall now define maps $\rho_X\colon\FV(X)\to D_X$, for $X\subseteq J$, as follows:

\begin{itemize}
\item
Denote by $\rho_\es\colon\FV(\es)\to D_\es=\two$ the unique isomorphism.

\item
For $i\in J$, denote by $\rho_i\colon\FV(\set{i})\to D_i=\three$ the unique $\cL$-ho\-mo\-mor\-phism sending~$a_i$ to~$1$.

\item
For $1\leq i<j\leq 3$, denote by $\rho_{i,j}=\rho_{j,i}\colon\FV(\set{i,j})\to D_{i,j}=\three^2$ the unique $\cL$-ho\-mo\-mor\-phism sending~$a_i$ to~$(2,1)$ and~$a_j$ to~$(1,2)$.

\item
Finally, denote by $\rho_{1,2,3}\colon\FV(J)\to D_{1,2,3}=\three^3\times\two$ the unique $\cL$-ho\-mo\-mor\-phism sending~$a_1$ to~$(2,2,1,1)$, $a_2$ to $(2,1,2,1)$, and~$a_3$ to $(1,2,2,1)$.

\end{itemize}

It is straightforward to verify that~$\vec{\rho}\eqdef\vecm{\rho_X}{X\subseteq J}$ is a natural transformation from the diagram~$\vec{F}$ of all~$\FV(X)$, for $X\subseteq J$, to~$\vec{D}$
(it is sufficient to check the required equations on the generators~$a_i$).
The diagrams~$\vec{F}$ and~$\vec{D}$, together with the natural transformation~$\vec{\rho}$, are represented in Figure~\ref{Fig:rho}.

\begin{figure}[htb]
 \[
 \xymatrixrowsep{2pc}\xymatrixcolsep{1.5pc}
 \def\labelstyle{\displaystyle}
 \xymatrix{
 &\FV(\set{1,2,3})\ar[rrrr]|{\,\rho=\rho_{1,2,3}\,}&&&&
 \three^3\times\two & \\
 \FV(\set{1,2})\ar[ur]\ar@/^1.4pc/[rrrr]|(.6){\,\rho_{1,2}\,} &
 \FV(\set{1,3})\ar[u]\ar@/_1.4pc/[rrrr]|(.6){\,\rho_{1,3}\,} &
 \FV(\set{2,3})\ar[ul]\ar@/^1.4pc/[rrrr]|(.4){\,\rho_{2,3}\,} &&
 \three^2\ar[ur]^{\ba} & \three^2\ar[u]^{\bb}
 & \three^2\ar[ul]_{\bc}\\
 \FV(\set{1})\ar[u]\ar[ur]\ar@/^1.4pc/[rrrr]|(.35){\,\rho_{1}\,}
 & \FV(\set{2})\ar[ul]\ar[ur]\ar@/_1.4pc/[rrrr]|(.6){\,\rho_{2}\,}
 &
 \FV(\set{3})\ar[u]\ar[ul]\ar@/^1.4pc/[rrrr]|(.3){\,\rho_{3}\,}
 &&
 \three\ar[u]^{\bbf}\ar[ur]_(.8){\!\!\!\bbf} &
 \three\ar[ul]^(.7){\bg\!\!\!}\ar[ur]_(.8){\!\!\!\bbf} &
 \three\ar[u]_(.7){\bg}\ar[ul]^(.75){\bg}\\
 &\FV(\es)\ar[rrrr]|{\,\rho_\es\,}\ar[ul]\ar[u]\ar[ur] &&&&
 \two\ar[ul]^{\be}\ar[u]^{\be}\ar[ur]_{\be} &
  }
 \]
\caption{The natural transformation~$\vec{\rho}$}
\label{Fig:rho}
\end{figure}
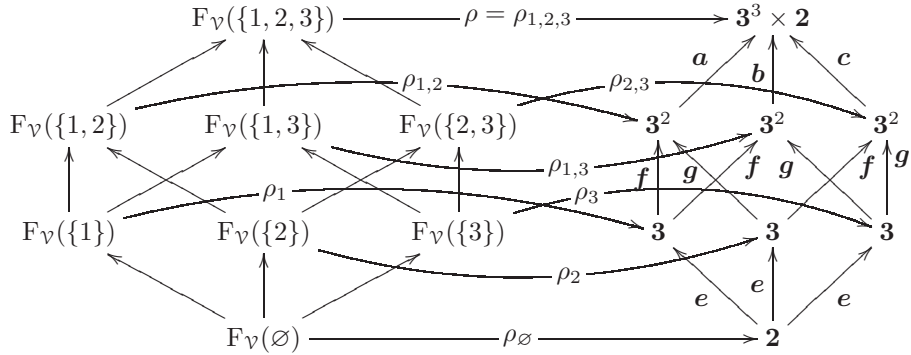

Setting $\rho\eqdef\rho_{1,2,3}$, we obtain, in particular, the equations
 \begin{equation}\label{Eq:arho12etc}
 \ba\circ\rho_{1,2}=\rho\res_{\FV(\set{1,2})}\,,\ 
 \bb\circ\rho_{1,3}=\rho\res_{\FV(\set{1,3})}\,,\ 
 \bc\circ\rho_{2,3}=\rho\res_{\FV(\set{2,3})}\,.
 \end{equation}
Denote by~$\pi\colon\FV(I)\twoheadrightarrow\FV(J)$ the unique $\cL$-ho\-mo\-mor\-phism such that
 \begin{equation}\label{Eq:defnpiai}
 \pi(a_i)\eqdef
 \begin{cases}
 a_i\,,&\text{if }i\in J\,,\\
 0\,,&\text{otherwise}, 
 \end{cases}
 \qquad\text{for every }i\in I\,.
 \end{equation}
The element $d_{i,j}\eqdef\pi(c_{i,j})$, for $i,j\in J$, belongs to~$\FV(J)$.
Moreover, it follows from~\eqref{Eq:defnpiai}, together with the assumptions on the~$c_{i,j}$\,, that the~$d_{i,j}$ satisfy the following relations:
 \begin{align}
 a_i&=(a_i\wedge a_j)\vee d_{i,j}\,,&&\text{whenever }i,j\in J\,;
 \label{Eq:ai2aijvdij}\\
 d_{i,j}\wedge d_{j,i}&=0\,,&&\text{whenever }i,j\in J\,;
 \label{Eq:dijdji0}\\
 d_{i,k}&\leq d_{i,j}\vee d_{j,k}\,,&&\text{whenever }i,j,k\in J
 \label{Eq:dijk}\,.
 \end{align}
Moreover, for distinct $i,j\in J$, the element~$c_{i,j}$ belongs to $\FV\pI{\Phi(\set{i,j})}$, that is, $c_{i,j}=\vt\vecm{a_k}{k\in\Phi(\set{i,j})}$ for an $\cL$-term~$\vt$.
Since~$\pi$ is an $\cL$-ho\-mo\-mor\-phism, $d_{i,j}=\vt\vecm{\pi(a_k)}{k\in\Phi(\set{i,j})}$.
By~\eqref{Eq:123indep} and~\eqref{Eq:defnpiai}, each~$\pi(a_k)$ is either~$0$ (if $k\notin J$) or belongs to $\set{a_i,a_j}$ (if $k\in J$).
It follows that $d_{i,j}\in\FV(\set{i,j})$.

Let $1\leq i<j\leq 3$.
Since $\rho_{i,j}(a_i)=(2,1)$ and $\rho_{i,j}(a_j)=(1,2)$, it follows from~\eqref{Eq:ai2aijvdij} and~\eqref{Eq:dijdji0} that
 \begin{align*}
 (2,1)&=(1,1)\vee\rho_{i,j}(d_{i,j})\,,\\
 (1,2)&=(1,1)\vee\rho_{i,j}(d_{j,i})\,,\\
 (0,0)&=\rho_{i,j}(d_{i,j})\wedge\rho_{i,j}(d_{j,i})\,,
 \end{align*}
which leaves the only possibility
 \begin{equation}\label{Eq:rhoijdij}
 \rho_{i,j}(d_{i,j})=(2,0)\text{ and }\rho_{i,j}(d_{j,i})=(0,2)\,.
 \end{equation}
By applying the homomorphisms~$\ba$, $\bb$, $\bc$, respectively, to the instances $(i,j)=(1,2)$, $(i,j)=(1,3)$, $(i,j)=(2,3)$, respectively, of~\eqref{Eq:rhoijdij}, we obtain, using~\eqref{Eq:arho12etc},
 \[
 \rho(d_{1,2})=(2,2,0,0)\,,\ 
 \rho(d_{1,3})=(2,2,0,1)\,,\ 
 \rho(d_{2,3})=(2,0,0,0)\,,
 \]
whence (projecting on the last coordinate) $\rho(d_{1,3})\nleq\rho(d_{1,2})\vee\rho(d_{2,3})$.
On the other hand, by applying the homomorphism~$\rho$ to the inequality~\eqref{Eq:dijk}, we obtain $\rho(d_{1,3})\leq\rho(d_{1,2})\vee\rho(d_{2,3})$; a contradiction.
\end{proof}

By applying Stone duality, we thus obtain the following result.

\begin{corollary}\label{C:FVInotlrep}
For every set~$I$, the spectrum~$\gO_I$ of the underlying distributive lattice of~$\FV(I)$ is a completely normal spectral space.
Moreover, if $\card{I}\geq\aleph_2$, then~$\gO_I$ cannot be embedded, as a spectral subspace, into any $\ell$-spectrum.
\end{corollary}


\providecommand{\noopsort}[1]{}\def\cprime{$'$}
  \def\polhk#1{\setbox0=\hbox{#1}{\ooalign{\hidewidth
  \lower1.5ex\hbox{`}\hidewidth\crcr\unhbox0}}} \def\cprime{$'$}
\providecommand{\bysame}{\leavevmode\hbox to3em{\hrulefill}\thinspace}
\providecommand{\MR}{\relax\ifhmode\unskip\space\fi MR }
\providecommand{\MRhref}[2]{%
  \href{http://www.ams.org/mathscinet-getitem?mr=#1}{#2}
}
\providecommand{\href}[2]{#2}

\end{document}